\documentclass{amsart}
\usepackage[english]{babel}
\usepackage{amsthm}
\usepackage{amsfonts}
\usepackage{amssymb}
\usepackage{graphicx}
\usepackage{float}
\usepackage[usenames,dvipsnames]{xcolor}
\usepackage{indentfirst}
 \usepackage{theoremref}
 \usepackage{todonotes}
 \usepackage{url}
\usepackage[normalem]{ulem}
\usepackage{comment}
\newtheorem{innercustomgeneric}{\customgenericname}
\providecommand{\customgenericname}{}
\newcommand{\newcustomtheorem}[2]{%
  \newenvironment{#1}[1]
  {%
   \ifdefined\crefalias\crefalias{innercustomgeneric}{#2}\fi
   \renewcommand\customgenericname{#2}%
   \renewcommand\theinnercustomgeneric{##1}%
   \innercustomgeneric
  }
  {\endinnercustomgeneric}%
  \ifdefined\crefname\crefname{#2}{#2}{#2s}\fi
}
\newtheorem{theorem}{Theorem}[section]
\newtheorem{corollary}{Corollary}[theorem]
\newtheorem{lemma}[theorem]{Lemma}

\newtheorem{proposition}[theorem]{Proposition}
\newcustomtheorem{customtheorem}{Theorem}
\theoremstyle{definition}
\newtheorem{definition}{Definition}[section]

\theoremstyle{remark}
\newtheorem*{remark*}{Remark}

\title{Rep-tiles}
 \author[Ryan Blair, Patricia Cahn, Alexandra Kjuchukova, Hannah Schwartz]{Ryan Blair, Patricia Cahn, Alexandra Kjuchukova,\\ And Hannah Schwartz}

\date{September 2025}

\begin{document}

\begin{abstract}
An $n$-dimensional rep-tile is a compact, connected submanifold of $\mathbb{R}^n$ with non-empty interior which can be decomposed into pairwise isometric rescaled copies of itself whose interiors are disjoint.  We show that every smooth compact $n$-dimensional submanifold of $\mathbb{R}^n$ with connected boundary is topologically isotopic to a polycube that tiles the $n$-cube, and hence is topologically isotopic to a rep-tile. It follows that there is a rep-tile in the homotopy type of any finite CW complex.  In addition to classifying rep-tiles in all dimensions up to isotopy, we also give new explicit constructions of rep-tiles, namely examples in the homotopy type of any finite bouquet of spheres.
\end{abstract}
\maketitle

\section{Manifolds which are rep-tiles}

\subsection{Main result.}
We prove that every compact codimension-$0$ smooth submanifold of $\mathbb{R}^n$  with connected boundary can be topologically isotoped to a polycube which tiles the cube $[0, 1]^n$.  As a consequence, any such manifold is topologically isotopic to a rep-tile, defined next. A {\it rep-tile}  $X$ is a codimension-0 subset of $\mathbb{R}^n$ with non-empty interior which can be written as a finite union $X=\bigcup_i X_i$ of pairwise  isometric sets $X_i,$ each of which is similar to $X$; and such that  $X_i, X_j$ have non-intersecting interiors whenever $i\neq j$. A rep-tile in $\mathbb{R}^n$ which is also homeomorphic to a compact smooth manifold will be called a {\it $n$-dimensional rep-tile}. 

Since every $n$-dimensional rep-tile has connected boundary (Lemma~\ref{lem:single-bdry}), as does every $n$-dimensional manifold that tiles the $n$-cube, our result proves that any submanifold of $\mathbb{R}^n$ which could potentially be homeomorphic to an $n$-dimensional rep-tile is in fact isotopic to one.  Thus, our work completes the isotopy classification of manifolds that tile the cube, and of $n$-dimensional rep-tiles, in all dimensions.

In Proposition~\ref{Prop:best} we also establish that, as one might expect,  
every submanifold $X$ of $\mathbb{R}^n$ is isotopic to one which is not a rep-tile. 
In light of our results, it may be said that only questions about the geometry of rep-tiles remain.

\subsection{A brief history of rep-tiles.}
Early sightings of rep-tiles were recorded in~\cite{gardner1963rep, golomb1964replicating}. Because $n$-dimensional rep-tiles tile $\mathbb{R}^n$, rep-tiles have been studied not only for their intrinsic beauty but also in connection with tilings of Euclidean space; see~\cite{gardner1977extraordinary} or~\cite{radin2021conway} for a discussion of the case $n=2$. A notable achievement was a non-periodic tiling of the plane by a rep-tile, due to Conway, which was later used to create the first example of a pinwheel tiling, i.e. one in which the tile occurs in infinitely many orientations~\cite{radin1994pinwheel}. The elegant 2-dimensional rep-tile portrayed in Figure~\ref{fig:L} was the building block in one of Goodman-Strauss's constructions of a hierarchical tiling of $\mathbb{R}^2$~\cite{goodman1998matching} and is also found in~\cite{thurston1989tilings}.

\begin{figure}[htbp]
    \centering
    \includegraphics[width=.5in]{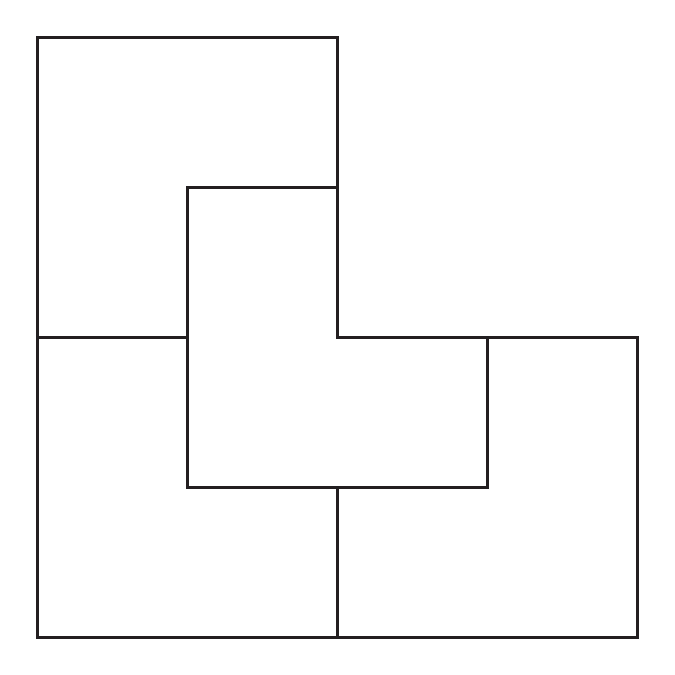}
    \caption{The ``chair" rep-tile.}
    \label{fig:L}
\end{figure}

The first planar rep-tile with non-trivial fundamental group was discovered by Gr\"unbaum, settling a question of Conway~\cite[C17]{unsolved}. 
 In 1998, Gerrit van Ophuysen found the first example of a rep-tile homeomorphic to a solid torus, answering a question by Goodman-Strauss  \cite{vanophuysen}. Tilings of $\mathbb{R}^3$ of higher genus were also constructed in~\cite{schulte1994space}. Tilings of $B^n$ by mutually isometric knots were constructed by~\cite{oh1996knotted}. 
Adams proved that any compact submanifold of $\mathbb{R}^n$ with connected boundary tiles it \cite{Adams95}.  Building on the above work, in 2021 came the homeomorphism classification of 3-dimensional rep-tiles.

\begin{theorem}\label{thm-BMR} \cite{blair2021three} A submanifold $R$ of $\mathbb{R}^3$ is homeomorphic to a 3-dimensional rep-tile if and only if it is homeomorphic to the exterior of a finite connected graph in $S^3$. 
\end{theorem}

The above implies that any 3-manifold which could potentially be homeomorphic to a rep-tile is indeed homeomorphic to one. This follows from Fox's re-embedding theorem \cite{fox1948imbedding}, which classifies compact 3-manifolds that embed in $S^3$, together with Lemma~\ref{lem:single-bdry}, which shows that a rep-tile has connected boundary. 

Our main result is Theorem~\ref{thm:main}, which completes the isotopy classification of manifold rep-tiles in all dimensions. In contrast with the above result, we do not rely on a classification of codimension-0 submanifolds of $\mathbb{R}^n$. Instead, we describe the isotopy from any submanifold of $\mathbb{R}^n$ which satisfies the hypotheses of the Theorem to a rep-tile.

\begin{theorem}\label{thm:main}
Let $R\subset \mathbb{R}^n$ be a compact smooth $n$-manifold with connected boundary. Then, $R$ is topologically isotopic to a rep-tile. 
\end{theorem}

\begin{corollary}\label{cor:homotopy} 
Let $X$ be a compact connected CW complex of dimension $n\geq 1$. Then $X$ is homotopy equivalent to a $(2n+1)$-dimensional rep-tile.
\end{corollary} 

\begin{proof}
    Suppose that $X$ is a compact connected CW complex of dimension $n$. Then, $X$ embeds in $\mathbb{R}^{2n+1}$, by the Nöbeling-Pontryagin Theorem~\cite[p.~125, Theorem~9]{denton1991general}. Let $R$ be a closed regular neighborhood of $X$ in  $\mathbb{R}^{2n+1}$. Then $R$ is a compact $(2n+1)$-manifold embedded in $\mathbb{R}^{2n+1}$. Moreover, $R$ has a single boundary component. Indeed, suppose $\partial(R)$ has two or more connected components $N_1, N_2, \dots N_k.$ 
  Since $N_j$ is a closed $2n$ manifold embedded in $\mathbb{R}^{2n+1}$, it is orientable, so $H_{2n}(N_j; \mathbb{Z})\cong \mathbb{Z}$. Moreover, since $k>1$, we have that $H_{2n+1}(R, N_j)=0$ for each $j=1, 2, \dots, k.$ Therefore, we see from the long exact sequence of the pair that the inclusion-induced map $i_\ast: H_{2n}(N_j)\to H_{2n}(R)$ is injective. But $R$ has the homotopy type of an $n$-complex, which is a contradiction. Therefore, by Theorem \ref{thm:main}, $R$ is isotopic to a rep-tile.
\end{proof}

The proof of Theorem~\ref{thm:main} describes a procedure for isotoping any codimension-0 smooth submanifold $R$ of $\mathbb{R}^n$ with connected boundary to a rep-tile. While the proof is constructive, in effect it is done without writing down any new rep-tiles.  In Section~\ref{sec:suspend} we therefore also give, for any $n\geq 0$, an explicit construction of a rep-tile homeomorphic to $S^n\times D^2$. This leads to an almost equally explicit construction of a rep-tile in the homotopy type of any finite bouquet of spheres. In particular, we can build explicit rep-tiles with non-vanishing homotopy groups in arbitrarily many dimensions.

The paper is organized as follows: in Section~\ref{sec:suspend} we construct a rep-tile homeomorphic to $S^n\times D^2$, presented explicitly as a union of cubes in $\mathbb{R}^{n+2}$, introduce the technique of cube swapping, show how to construct rep-tiles homotopy equivalent to wedges of spheres, investigate suspensions of rep-tiles, and construct rep-tiles with arbitrary footprints. Section~\ref{sec:proof} is where we prove the main theorem.

\subsection{Rep-tiles and tilings of Euclidean space}
Rep-tiles induce self-similar tilings of Euclidean space. Thus, they can potentially be used to construct non-periodic and aperiodic tilings of the plane and higher-dimensional Euclidean spaces. Self-similar tilings have connections to combinatorial group theory~\cite{conway1990tiling}, propositional logic~\cite{wang1960proving, berger1966undecidability, robinson1971undecidability} (where some of the questions in the field originated), and dynamical systems~\cite{thurston1989tilings}, among others.  Our rep-tiles give new self-similar tilings of $\mathbb{R}^n$ by tiles with interesting topology. Additionally, in our proof of Theorem \ref{thm:main} we show that every compact smooth $n$-manifold with connected boundary in $\mathbb{R}^n$ is topologically isotopic to a polycube that tiles a cube. In \cite{golomb1966TilingHierarchy}, Golomb developed a hierarchy for polycubes that tile $\mathbb{R}^n$, and tiling a cube is the most restrictive level of his hierarchy. Hence all compact smooth $n$-manifolds with connected boundary lie in the most restrictive level of Golomb's hierarchy, up to isotopy.

\section{Explicit rep-tiles in all dimensions} 
\label{sec:suspend}

In this section, an $n$-dimensional {\it polycube} is a union of unit $n$-cubes whose vertices lie on the integer lattice in $\mathbb{R}^n$.  We will be repeatedly using the fact that a polycube that tiles a cube is a rep-tile (see Lemma~\ref{mantra}).  
In this section we will realize the homotopy type of certain $m$-manifolds $X$ as rep-tiles by the following procedure. We will construct a polycube $R\subset \mathbb{R}^m$ such that $R\simeq X$ (where $\simeq$ denotes homotopy equivalence) and such that two copies of $R$, related by a rotation, tile the $m$-dimensional cube. In the case where $X$ has the homotopy type of a sphere, $X\simeq S^n$,  the rep-tile $R$ we build is homeomorphic to a trivial 2-dimensional disk bundle on the sphere, $R\cong S^n\times D^2$. (This demonstrates that rep-tiles can have non-trivial $\pi_n$ for all $n\geq 0,$ answering Conway's and Goodman-Strauss's question in dimensions three and higher.) It is a fairly straightforward consequence that finite wedges of spheres of different dimensions can be built similarly.

\subsection{Stacks of cubes}

\label{sec:stacks} 

A {\it stack of $n$-cubes with stacking direction $x_n$} is an $n$-dimensional polycube $S\subset \mathbb{R}^n$ such that:  (1) All $n$-cubes in $S$ lie above the hyperplane $x_n=0$ (that is, every point in $S$ has non-negative $x_n$ coordinate), and (2) for every $n$-cube in $S$ that does not have a face contained in $x_n=0$, there is another $n$-cube of $S$ directly below it (where height is measured by the $x_n$-axis).

Let the subspace of $\mathbb{R}^n$ determined by $x_n=0$ have the standard tiling by $(n-1)$-cubes induced by the integer lattice in $\mathbb{R}^n$. Given $S$, a stack of $n$-cubes with stacking direction $x_n$, we consider its projection to the hyperplane $x_n=0$, which we call its {\it footprint}. By the definition of a stack of cubes, we can think of $S$ as consisting of columns of $n$-cubes lying above each $(n-1)$-cube in its footprint $\mathcal{F}_S$, which is itself an $(n-1)$-dimensional polycube. In other words, the homotopy type of $S$ is determined by $\mathcal{F}_S$; and $S$ itself is determined by $\mathcal{F}_S$, together with integer labels in each $(n-1)$-cube of $\mathcal{F}_S$, specifying the height of the column of $n$-cubes which lie above it. Therefore, we can describe $S$ by such a labeled footprint. Figure~\ref{fig:stack_cubes_0} illustrates a 2-dimensional stack of cubes (left) and its description via a labeling on its 1-dimensional footprint (right).

The image of such a stack of cubes under an isometry of $\mathbb{R}^n$ is also called a stack of cubes, with the image of $x_n$ under the isometry being the stacking direction.

\begin{figure}
    \centering
    \includegraphics[width=4in]{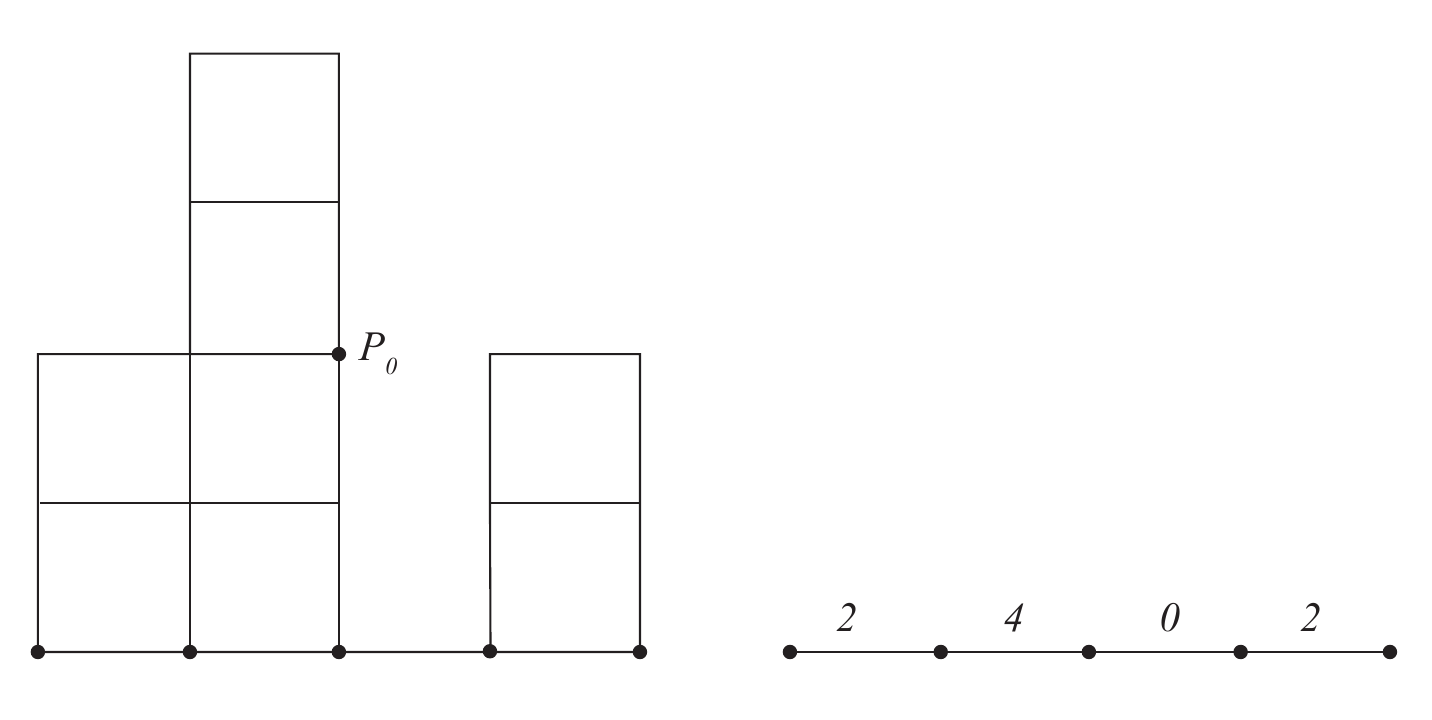}
    \caption {A stack of cubes (left) and its labeled footprint (right). 
    This polycube and its image under rotation by $\pi$ about $P_0$ tile $[0,4]^2$. Thus, the polycube is a rep-tile.}
    \label{fig:stack_cubes_0}
\end{figure}

\subsection{Rep-tiles homotopy equivalent to $S^n$.}\label{sec:spheres}

We use cube-stacking notation as above to describe a rep-tile homeomorphic to $S^n\times D^2$ for all $n\geq 0$. This description is a simplification, suggested by Richard Schwartz~\cite{Schwartz-email}, of the construction given in~\cite{blair2024rep}. We define a stack of cubes $S\subset [0,4]^{n+2}$ as follows.  The footprint $\mathcal{F}_S$ is a polycube in $[0,4]^{n+1}$.
\begin{figure}
    \centering
    \includegraphics[width=\linewidth]{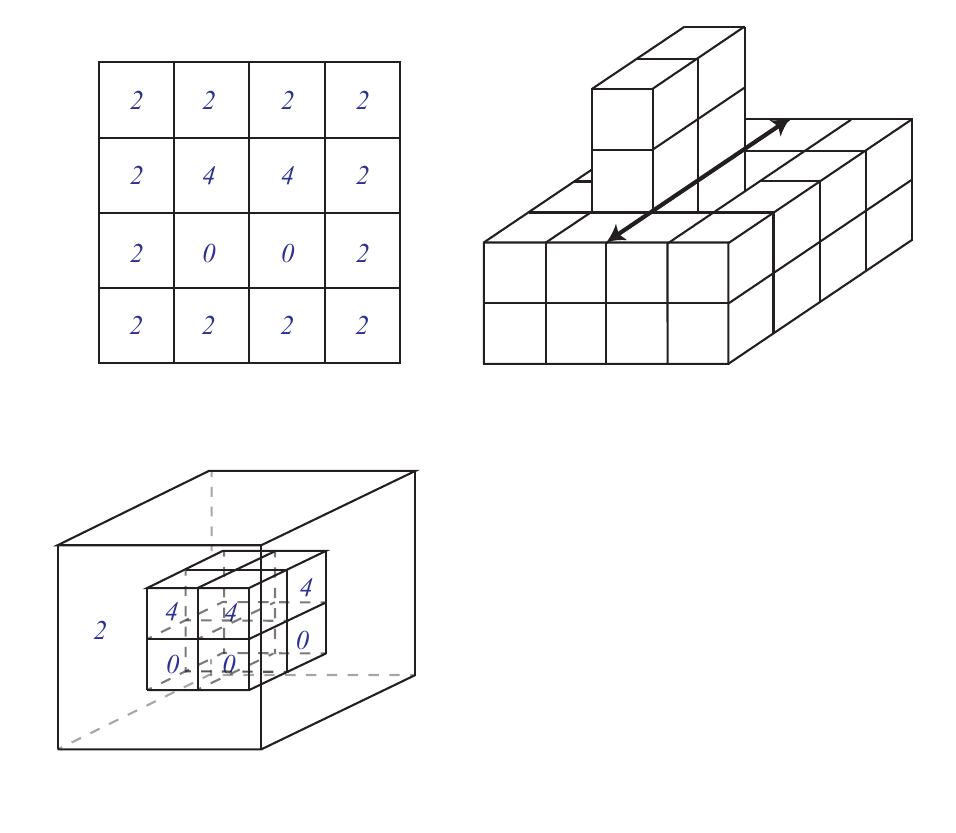}
    \caption{ Top : Footprint of a 3-dimensional rep-tile homeomorphic to $S^1\times D^2$ (left), and its corresponding stack of cubes (right), rotated by 90 degrees for visualization. Bottom: Footprint of a 4-dimensional rep-tile homeomorphic to $S^2\times D^2$ (left), and space to imagine the corresponding stack of cubes (right).}
    \label{fig:sphere_stack_cubes}
\end{figure}

Throughout the following discussion, the reader should refer to Figure~\ref{fig:sphere_stack_cubes}.   Define the {\it core}, denoted $\mathcal{C}$, of $[0,4]^{n+1}$ to be the union of unit cubes in the standard integer-lattice tiling of $[0,4]^{n+1}$ containing the point $(2,...,2)$.  The {\it shell} of $[0,4]^{n+1}$ is $\overline{[0,4]^{n+1}\setminus \mathcal{C}}.$  To create the labeled footprint $\mathcal{F}_S$ of our stack of cubes $S$, we first partition $\mathcal{C}$ into two halves: $\mathcal{C}^+$, those containing cubes with $x_{n+1}$-coordinate at least 2; and $\mathcal{C}^-$, those containing cubes with  $x_{n+1}$-coordinate less than 2.  Finally, we label each cube in $\mathcal{C}^+$ with a 4, and each cube in $\mathcal{C}^-$ with a 0.  All cubes in the shell are labeled 2. (We recall that the label of each $(n+1)$-cube in the footprint indicates the height of the column of $(n+2)$-cubes stacked on top of it.) Observe that $\mathcal{F}_S$, which consists of all unit cubes in $[0,4]^{n+1}$ with nonzero label, is homeomorphic to the shell, which is in turn homeomorphic to $S^n\times D^1$. Similarly, the stack of cubes $S$ determined by this labeling is homeomorphic to $\mathcal{F}_S\times I\cong S^n\times D^2$.

Next we show that $S$ is a rep-tile.  Let $r_\pi:\mathbb{R}^{n+2}\rightarrow \mathbb{R}^{n+2}$ denote rotation by $\pi$ about the $n$-plane which is the intersection of $x_{n+2}=2$ and $x_{n+1}=2$.  Observe that the closure of the complement of $S$ in $[0,4]^{n+2}$ is also a stack of cubes, with stacking direction $-x_{n+2}$, is isometric to $S$, and in particular, is the image of $S$ under $r_\pi.$  As $S$ and $r_\pi (S)$ tile the cube $[0,4]^{n+2}$, and since $S$ is a union of cubes, $S$ is a rep-tile. 

\subsection{Cube swapping.} \label{sec:ball-number}
We note that there is a lot of flexibility regarding the heights of columns in the construction of a rep-tile $S\cong S^n\times D^2$ given above. Consider any column $H$ in $S$ of height $h\in \{1, 2, 3\}$. Let $H'$ denote the column of $S$ which shares a footprint with $r_{\pi}(H)$. Since $H'\cup r_{\pi}(H)$ form a column of height 4, the heights of $H$ and $H'$ add up to 4. Moreover, unit cubes can be traded between $H$ and $H'$ while preserving the property that the resulting polycube and its image under $r_{\pi}$ tile $[0,4]^{n+2}$. As long as both columns remain of height strictly between 0 and 4 and their heights add up to 4, this swap preserves both the homeomorphism type of $S$ and the property that two copies of $S$ tile a cube.

More generally, let $R$ be any non-empty $n$-dimensional polycube in $\mathbb{R}^n$. Let $G$ be a group of isometries of $\mathbb{R}^n$ such that the orbit of $R$ under $G$ tiles a cube $\mathcal{C}$. (As before, this implies that $R$ is a rep-tile.) Let $u$ denote any unit cube contained in $R$ and let $g$ be an arbitrary element of $G$. Denote by $gu$ the image of $u$ under $g$. We note that $R':=\overline{R\backslash u}\cup gu$ is also a polycube whose orbit under $G$ is $\mathcal{C}$. Hence, $R'$ is also a rep-tile. We will refer to this move as a {\it cube swap}. (Aside: the fact that performing a cube swap on a polycube that tiles an $n$-cube produces another polycube that tiles an $n$-cube does not depend in an essential way on the fact that $u$ is a $n$-cube. More complicated pieces could be swapped as well, preserving the tiling property.)  Cube swapping, which was inspired by work of Adams~\cite{Adams95, Adams97}, turns out to be a powerful tool for building rep-tiles, as we will see in Section~\ref{sec:proof}. To be precise, a version of the cube swap -- one which involves an action of a group of order $2^m$ on $[-1, 1]^{m}$ and trading multiple groups of cubes $U_i$ simultaneously across their individual orbits -- is the key idea in the Proof of Theorem~\ref{thm:main}. The second main ingredient in the proof is this: a priori, $R'$ might not have a clear relationship to $R$; to guarantee that $R'$ is homeomorphic to or isotopic to $R$, care must be taken in the choice of group action and the choice of $U_i$.

\subsection{Rep-tilean bouquets}\label{sec:bouquets}

Let $P_{n}=\{\textbf{x}\in \mathbb{R}^{n+2}|x_{n+2}=x_{n+1}=2\}$. The construction in Section \ref{sec:spheres} has produced stacks of $(n+2)$-dimensional cubes in $[0,4]^{n+2}$ with the following useful properties: 
\begin{enumerate}
    \item Each polycube intersects $x_1=0$ in an $(n+1)$-ball equal to $\{0\}\times [0,4]^n \times [0,2]$ and intersects $x_1=4$ in an $(n+1)$-ball equal to $\{4\}\times [0,4]^n \times [0,2]$;
    \item the polycube and its image under rotation by $\pi$ about $P_{n}$ tile $[0,4]^{n+2}$.
\end{enumerate}

Note that any two such polycubes of the same dimension $R_1$ and $R_2$ can be placed side-by-side in the $x_1$ direction so that $R_1$ is contained in $0\leq x_1 \leq 4$ and $R_2$ is contained in $4\leq x_1\leq8$. For example, place two copies of the stack of cubes in the top of Figure~\ref{fig:sphere_stack_cubes} back-to-back. In this configuration $R_1\cap R_2=\{4\}\times [0,4]^n \times [0,2] \cong B^{n+1}$. Thus, $R_1\cup R_2$ has the homotopy type of the wedge $R_1\vee R_2$; and, after rescaling in the $x_1$ direction and subdividing the integer lattice, it too satisfies the conditions (1) and (2) above. 

Now consider $S_{m}$ and $S_k$, two of the rep-tiles constructed in Section \ref{sec:spheres} of dimension $m$ and $k$ respectively. If $m\leq k$, then $S_m\times D^{k-m}$ can be embedded in $[0,4]^{k+2}$ so that conditions (1) and (2) hold. By stacking $S_k$ and this embedding of $S_m\times D^{k-m}$ as in the previous paragraph, we construct a rep-tile in the homotopy type of $S^m\vee S^k$, itself capable of becoming part of a further rep-tilean wedge. By iterating this process, rep-tiles in the homotopy type of any finite wedge of spheres can be constructed.

\subsection{Suspending Rep-Tiles}\label{sec:susp} 
    Let $r_\pi$  be an order 2 rotation about some $(n-2)$-subspace in $\mathbb{R}^n$.  We note that if $R$ is any connected $n$-dimensional stack of cubes such that two copies of $R$, related by $r_\pi$, tile an $n$-cube, then $R$ can be used to construct an $(n+1)$-dimensional rep-tile in the homotopy type of the suspension of $R$. We sketch this construction with a specific choice of coordinates below. For clarity, we assume that $R\cup r_\pi(R)$ tile the cube $[0, 4]^{n}$. 
    
    Let $R$ denote any $n$-dimensional stack of unit cubes which has the property that $R$ and its image under under $r_\pi$ tile $[0, 4]^{n}$.
    (For instance, $R$ could be one of the rep-tiles in the homotopy type of a wedge of spheres that we previously constructed.) Because $R\cup r_\pi(R)=[0, 4]^{n},$ we know that $r_\pi$ takes cubes at height 4 (with respect to the stacking direction) to holes at height zero; and vice-versa. In particular, $R$ contains as many cubes at height 4 as it has unit-cube-sized holes at height 0. Therefore, we may suspend $R$ as by the following steps
    \begin{enumerate}
        \item embed $R\times[0,4]$ into $\mathbb{R}^{n+1}$;
        \item cubify in the natural way, writing $R\times[0,4]$ as a union of unit $(n+1)$-cubes of the form $(n\text{-cube in }R)\times [i, i+1]$;
        \item move all height-4 cubes in $R\times[0, 1]$ to fill all holes at height zero in that slice; \label{item-push1}
        \item repeat the last step in $R\times[3, 4]$.\label{item-push2}
    \end{enumerate}  
    
    Crucially, steps~\ref{item-push1} and~\ref{item-push2} constitute cube swaps (see Section~\ref{sec:ball-number}). This guarantees that the resulting polycube is still a rep-tile. Moreover, since the slice $R\times[i, i+1]$ is a stack of cubes, filling all cubes that correspond to ``height-0 holes" in $R\times [i,i+1]$ (that is, those holes in $R\times[i, i+1]$ which are height-0 holes in $R$ crossed with $[i, i+1]$) turns $R\times[i, i+1]$ into a ball. Therefore, as before, the cube swaps performed in the first and last slices of $R\times[0, 4]$ have the effect, up to homotopy, of contracting each of the ends of $R\times [0, 4]$ to a point. This completes the suspension of $R$.
    Figures~\ref{fig:stack_cubes_1_falling} and~\ref{fig:stack_cubes_2} illustrate the suspensions of rep-tiles homeomorphic to $S^0\times D^2$ and $S^1\times D^2$, respectively. 
    
Let $H$ and $H'$ denote any pair of columns in Figures~\ref{fig:stack_cubes_1_falling} or~\ref{fig:stack_cubes_2} which trade a cube during the cube-swapping operations. Specifically, say $H'$ is height 0 and the top cube of $H$ is moved to $H'$ during the cube swap. Now suppose next highest cube of $H$ (now at height 3) is also moved to column $H'$. This would also constitute a valid cube swap, since the unit cube remains within its orbit under the rotation. Executing this additional swap between all such pairs has the effect that all columns of heights 3 and 1  become columns of height 2. The result would be the $S^n\times D^2$ rep-tile constructed in Section~\ref{sec:spheres}. Put differently, rep-tiles homeomorphic to $S^n\times D^2$ can also be obtained from the $S^0\times D^2$ rep-tile in Figure~\ref{fig:stack_cubes_0} inductively, via a sequence of suspensions and cube swaps. For more details on this approach, see Section~2 of~\cite{blair2024rep}.

\begin{figure}
    \centering
    \includegraphics[width=.75\linewidth]{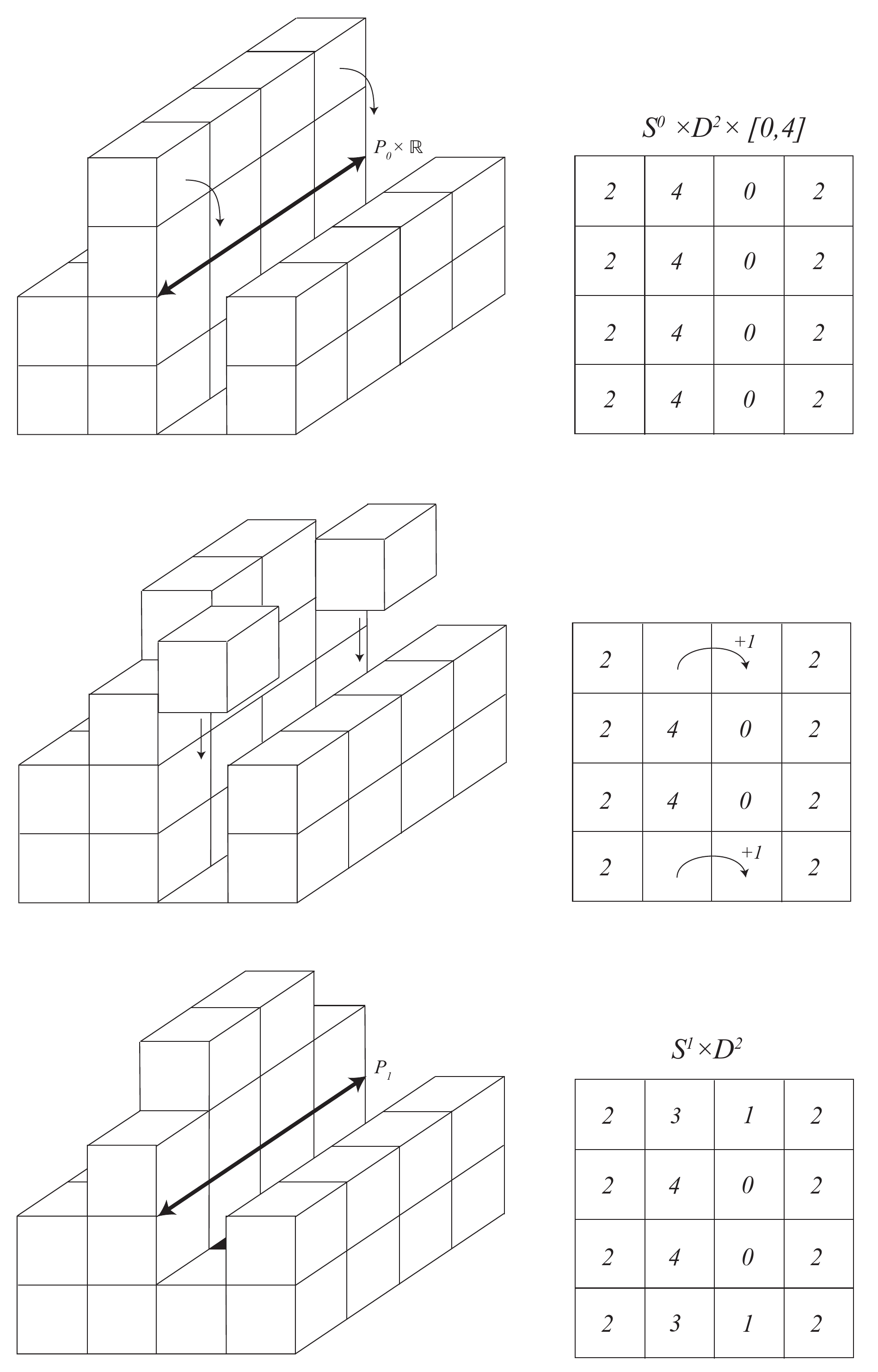}
   \caption{Cube swaps in a 3-dimensional rep-tile. The swap effectuates the suspension of a polycube representation of $S^0\times D^2$ to obtain a polycube representation of $S^1\times D^2$. Top: $S^0\times D^2\times [0, 4]\cong S^0\times D^3$. Middle: A cube swap which ensures that the first and last slices become disks. Bottom: the union of the four layers is a rep-tile homeomorphic to $S^1\times D^2$, the result of the suspension. A further cube swap between the same pairs of columns would result in the $S^1\times D^2$ rep-tile given in Section~\ref{sec:spheres}.}
    \label{fig:stack_cubes_1_falling} 
\end{figure}

\begin{figure}
    \centering
    \includegraphics[width=\linewidth]{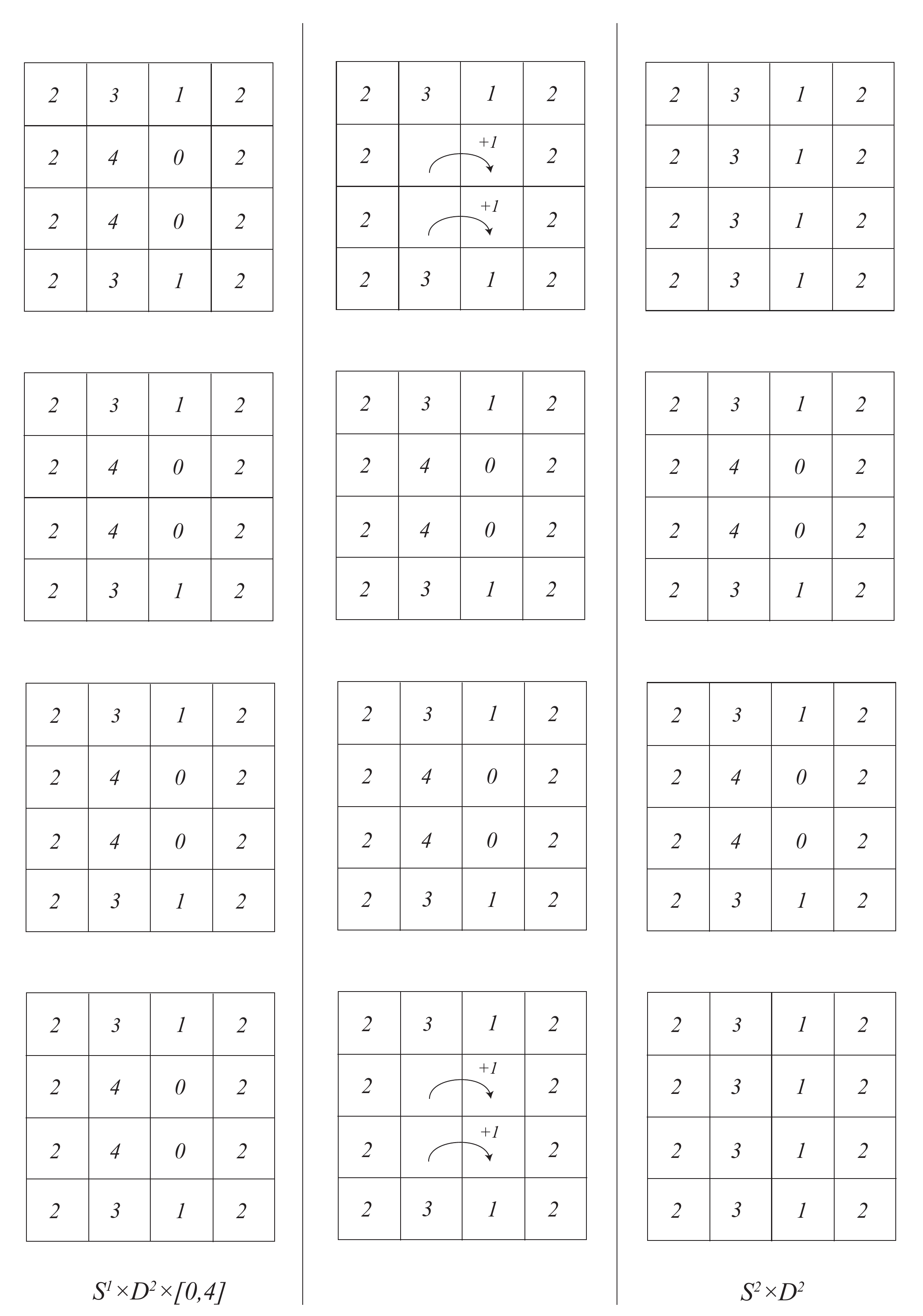}
    \caption{The left and right columns represent the labeled footprints of $4$-dimensional stacks of cubes. Taken together, the three columns depict the process of suspension from $S^1\times D^2$ to $S^2\times D^2$. \\ {\it Left column:} four layers of $S^1\times D^2\times [i, i+1]$, combining to form $S^1\times D^2\times [0, 4]$. {\it Middle column:} cube swaps occur in the first and fourth {slices}. {\it Right column:} bottom {slice}: $D^3\times [0,1]$, second  {slice}: $S^1\times D^2\times[1,2]$; third {slice}: $S^1\times D^2\times[2,3]$; fourth {slice}: $D^3\times[3,4]$. The union of the four {slices} is the suspended rep-tile.\\
    Note that by a further cube swap we could replace all 3's and all 1's by 2's. This would produce another rep-tile homeomorphic to  $S^2\times D^2$, namely the one described in Section~\ref{sec:spheres}.}
    \label{fig:stack_cubes_2}
\end{figure}

\subsection{Rep-tiles with arbitrary footprints} The following was observed by Richard Schwartz~\cite{Schwartz-email} while perusing the first version of our article.
\begin{proposition}\label{prop:Richard}\cite{Schwartz-email}
    There is an $n$-dimensional rep-tile in the homotopy type of any compact polycube in $\mathbb{R}^{n-1}$.
\end{proposition}

This result, together with the existence of cubifications for smooth codimension-0 submanifolds of $\mathbb{R}^n$ (see Section~\ref{sec:cubification}) can be used to prove a version of Corollary~\ref{cor:homotopy}. Specifically, we see that it is possible to realize the homotopy type of any compact $n$-dimensional CW complex as a $(2n+2)$-dimensional rep-tile $R$, without appealing to Theorem~\ref{thm:main}. The present approach uses an extra dimension; but it is rather explicit (given a polycube footprint to start with) and has the advantage that just 2 copies of $R$ can tile the $(2n+2)$-cube.

\begin{proof}[Proof of Proposition~\ref{prop:Richard}]
    
    We first observe that for any compact $(n-1)$-polycube $P$ there is a positive even integer $k$ such that $P$ is isotopic to an $(n-1)$-polycube $P'$ in $[0,k+2]^{n-1}$ such that $P'$ contains all unit cubes in $[0,k+2]^{n-1}$ whose smallest $x_{n-1}$-coordinate is equal to $0$. (To see this, begin by translating $P$ so that it is contained in $[0,k]^{n-1}$. Then, apply the following sequence of isotopies: shift $P$ at least two units away from the $x_{n-1}=0$ hyperplane in the positive $x_{n-1}$ direction; then grow a (cubical) finger out of $P$ until it touches $x_{n-1}=0$; then add the cubes whose union is $ [0,k+2]^{n-2}\times [0,1]$ to $P$.)

Next create an $n$-dimensional stack of cubes $S$ whose footprint is a polycube in $[0,k+2]^{n-2}\times [-(k+2),(k+2)]$, namely the boundary connected sum of $P'$ with $[0,k+2]^{n-2}\times [-(k+2),0]$.  We shall label the $(n-1)$-cubes contained in $[0,k+2]^{n-2}\times [-(k+2),(k+2)]$ to indicate the height of the corresponding column. In this manner, we will obtain the desired stack of $n$-cubes $S$ in the homotopy type of $P$. The $(n-1)$-cubes in $P'$ are labeled $k+2$. All $(n-1)$-cubes which are contained in $[0,k+2]^{n-2}\times [0,(k+2)]$ but not in $P'$ are labeled 0.  Let $r$ denote reflection in $\mathbb{R}^{n-1}$ about the plane $x_{n-1}=0$.  Cubes in $[0,k+2]^{n-2}\times [-(k+2),0]$ that are contained in $r(P')$ are labeled $k+2$. Remaining cubes are labeled $2k+4$. See Figure \ref{fig:arbitrary_footprint}.

Let $\rho$ be rotation about the $(n-2)$-plane in $\mathbb{R}^n$ determined by $x_{n-1}=0$ and $x_n=k+2$.
Next we observe that the sum of the labels of each unit cube in $[0,k+2]^{n-2}\times [-(k+2),(k+2)]$ and its reflection $r$ about $x_{n-1}=0$ sum to $2k+4$. It follows that the stack of cubes $S$ determined by this labeling, together with $\rho(S)$, tile $[0,k+2]^{n-2}\times[-(k+2),k+2]\times [0,2k+4]$. Then, after rescaling, two isometric copies of $S$ tile an $n$-cube. This produces a rep-tile in the homotopy type of the original footprint, $P$, as desired. 
\end{proof}

\begin{figure}
    \centering
    \includegraphics[width=0.5\linewidth]{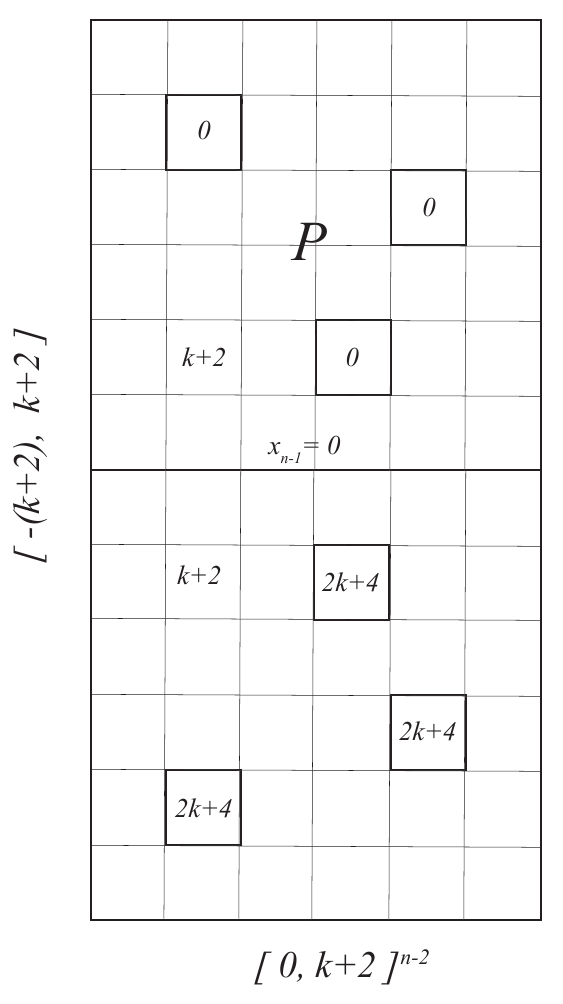}
    \caption{$P$ denotes a compact polycube in $\mathbb{R}^{n-1}$, embedded as a proper subset of a $(n-1)$-dimensional cube (pictured as the top $6\times 6$ square) in the hyperplane $x_n=0$ in $\mathbb{R}^{n}$. The figure is a schematic for constructing an $n$-dimensional rep-tile in the homotopy type of $P$. Specifically, the rep-tile's footprint is the pictured stack of cubes, which is isotopic to $P$. Each unlabeled box has $k+2$ cubes stacked on top of it; the heights of other stacks are as written. Remark that the bottom half of the picture is a stack of cubes homeomorphic to $B^{n}$; its footprint is an $(n-1)$-dimensional cube. This ball is added to $P$ to ensure symmetry.}
    \label{fig:arbitrary_footprint}
\end{figure}

\section{All is rep-tile}
\label{sec:proof}

We will denote the standard integer lattice in $\mathbb{R}^n$, consisting of all points in $\mathbb{R}^n$ with integer coordinates, by $\mathcal{Z}^n$.  This lattice induces a cell structure $\mathcal{C}(\mathcal{Z}^n)$ on $\mathbb{R}^n$, whose $k$-cells are the $k$-facets of unit cubes with vertices in $\mathcal{Z}^n$.  

We will also work with subdivisions of this lattice, and refer to the closed $n$-cells in any such decomposition as {\it atomic cubes}. The size of an atomic cube will depend on the subdivision used. Precisely, suppose $\lambda>0$ and let $f_{\lambda}:\mathbb{R}^n\rightarrow \mathbb{R}^n$ denote the scaling function given by $f(x)=\lambda x$. Let $\mathcal{Z}_{\lambda}^n=f(\mathcal{Z}^n)$, and let $\mathcal{C}(\mathcal{Z}_\lambda^n)$ denote the corresponding cell structure.

\begin{definition}\label{def:polycube}
 An {\it $n$-dimensional polycube} is a submanifold of $\mathbb{R}^n$ that is isometric to a finite union of atomic cubes in $\mathcal{C}(\mathcal{Z}^n_{\lambda})$ for some $\lambda\in\mathbb{R}$.    
\end{definition}

\begin{definition} \label{tiling.def}
   A compact $n$-manifold $T$ is said to \emph{$k$-tile}  a subset $A\subseteq \mathbb{R}^n$ if $A=\cup_{i=1}^{k} T_i$ such that $T_i$ is isometric to $T_j$ for all $i$ and $j$, and $int(T_i)\cap int(T_j)=\emptyset$ for all $i\neq j$.
\end{definition}

\begin{lemma} \label{mantra}
 Let $R$ be an $n$-dimensional polycube that tiles a cube $C$. Then, $R$ is a rep-tile. 
\end{lemma}

\begin{proof}
    By identifying each atomic cube in the polycube decomposition of $R$ with $C$, we can tile each cube in $R$ with a finite number of pairwise isometric manifolds, each of which is similar to $R$.  We have thus tiled $R$ by rescaled copies of $R$.
\end{proof}

In particular, a polycube that tiles the cube must have connected boundary, which follows from the following Lemma.

\begin{lemma}\label{lem:single-bdry}
    Let $X^n$ be a manifold which is homeomorphic to an $n$-dimensional rep-tile. Then $\partial(X)$ is non-empty and connected.
\end{lemma}
\begin{proof}
    Since $X^n$ is a homeomorphic to a rep-tile, we have that $X^n$ embeds in $\mathbb{R}^n$. Hence, $\partial(X)\neq\emptyset$. The proof that $\partial(X)$ is connected when $n=3$ is given in~\cite[Theorem~4.2]{blair2021three} and works without modification in all dimensions. 
\end{proof}

The following proposition is a variant of the well-known fact that smooth manifolds can be approximated by PL manifolds.

\begin{proposition}\label{Prop:Smooth_to_poly}
    Let $R\subset\mathbb{R}^n$ be a compact smooth $n$-manifold. Then, $R$ is topologically isotopic to a $n$-dimensional polycube. 
\end{proposition}

\begin{proof}
   Recall the elementary measure theory result that every open subset $\mathcal{O}$ of $\mathbb{R}^n$ can be written as a
countable union of closed $n$-cubes with disjoint interiors. In particular, there is a sequence of $n$-dimensional polycubes $P_1 \subset P_2 \subset P_3 ... \subset \mathcal{O}$ which limit to $\mathcal{O}$ with the property that $P_i$ is a union of $n$-cubes of side-length $(\frac{1}{2})^{i-1}$. See Theorem 1.4 of \cite{stein2009real} for details regarding the construction of the $P_i$. Let $\mathcal{O}=int(R)$. Since $R$ is compact, each of the $P_i$ are a union of finitely many cubes. When $i$ is sufficiently large, one can use the fact that $\partial R$ is smooth to build a topological isotopy from $P_i$ to $R$. We omit the details of this argument since they are elementary and somewhat lengthy.
\end{proof}

We recall our main theorem below. 
\begin{customtheorem}{\ref{thm:main}}
Let $R\subset \mathbb{R}^n$ be a compact smooth $n$-manifold with connected boundary. Then, $R$ is topologically isotopic to a rep-tile. 
\end{customtheorem}

Our main theorem is a consequence of the following.

\begin{theorem}\label{tile_cube.thm}
    Let $R\subset\mathbb{R}^n$ be a compact smooth $n$-manifold with connected boundary. Then, $R$ is topologically isotopic to a $n$-dimensional polycube $R^*$ which $2^n$-tiles a cube. 
\end{theorem}

A key step in the proof that any $R\subseteq \mathbb{R}^n$ satisfying the hypotheses of Theorem~\ref{thm:main} is isotopic to a rep-tile is to decompose $\overline{C^n\backslash R}$, the closure of the complement of $R$ in an $n$-cube, into a union of closed $n$-balls with non-overlapping interiors. Given a manifold $X^n$, the smallest number of $n$-balls in such a decomposition of $X$ is called the {\it ball number} of $X$, denoted $b(X)$. Upper bounds on the ball number of a manifold in terms of its algebraic topology have been found by Zeeman~\cite{zeeman1963seminar} and others~\cite{luft1969covering, KT76, singhof1979minimal}. 
 We rely on the following.

\begin{theorem}\label{theorem:ballnumberboundary}[2.11 of \cite{KT76}]
    Let $M^{n}$ be a connected compact PL $n$-manifold with non-empty boundary. Then $b(M)\leq n$. 
\end{theorem}

\begin{figure}[htbp]
    \centering
    \includegraphics[width=1.5in]{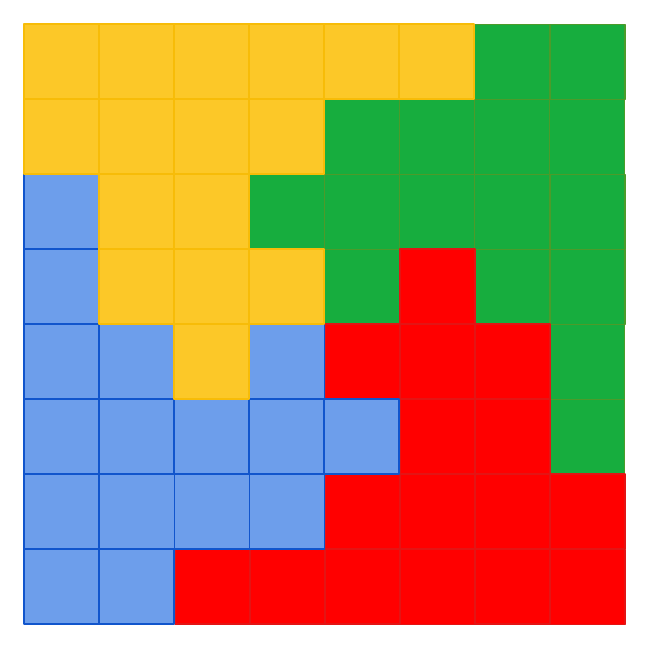}
    \caption{The green topological disk $R^*$, a rep-tile constructed from the top right $4\times 4$ square by cube swapping, tiles the $8\times 8$ square (which we may regard as a subdividison of $[-1,1]\times [-1,1]$). }
    \label{fig:2dschematic}
\end{figure}

\subsection{Overview of the proof of Theorem~\ref{thm:main}}

The main ingredient is Theorem ~\ref{tile_cube.thm}, which we prove using a strategy we refer to as a {\it cube swap}.  To start, $R$ is smoothly embedded in $C^n=[0,1]^n$ so that $R\cap\partial C^n =\emptyset$. In turn, the unit cube $C^n$ sits inside the cube $\boxplus=[-1,1]^n$. Since $R$ is disjoint from $\partial C^n$ and has a single boundary component, $C^n\setminus R$ is connected. By Theorem ~\ref{theorem:ballnumberboundary}, we may decompose $\overline{C^n\setminus R}$ into $n$ $n$-dimensional balls $B_1,\dots, B_n$.\footnote{If $b(C^n\setminus R)<n$, one could use fewer balls here and tile the cube with fewer copies of $R$, but we use $n$ balls for simplicity in the proof of the main theorem.}   After a homotopy of $C^n$ which restricts to an isotopy on each piece of the decomposition $\{R,B_1,\dots, B_n\}$ of $C^n$, we ensure that the pieces of this decomposition intersect an $(n-1)$-disk on $\partial{C}^n$ as shown in Figure~\ref{fig:claw}, in what we call a {\it taloned pattern}. The defining features of taloned patterns include: there is an $(n-1)$-disk on $\partial C^n$ such that $R$ and each of $B_1,\dots B_n$ intersect that disk in an $(n-1)$-ball and intersect the boundary of the disk in an $(n-2)$-ball; the $(n-1)$-balls $B_i$ are disjoint inside this disk; and $R$ is adjacent to each ball $B_i$ in this disk. (See Section~\ref{sec:talons} for the formal definition.) The homotopy used to create the taloned pattern is achieved in Lemmas~\ref{connected_bdry_v2.lem} and~\ref{lem:ClawsExist} below.  We then isotope $C^n$ so that the $(n-1)$-disk which constitutes the taloned pattern of Figure~\ref{fig:claw} is identified with the union of faces of $C^n=[0,1]^n$ whose interiors lie in the interior of $\boxplus=[-1,1]^n$, with certain additional restrictions. These restrictions guarantee that certain rotated copies of the $B_i$ contained in cubes adjacent to $[0,1]^n$ in $\boxplus=[-1,1]^n$ are disjoint, allowing us to form the boundary connected sum of $R$ with these balls without changing the isotopy class of $R$.  Indeed, we give a family of rotations $r_k$, $1\leq k \leq \lfloor{n/2}\rfloor$, together with one additional rotation $f$ if $n$ is odd, such that the orbit of $C^n$ under these rotations tiles $\boxplus$.  By taking the boundary sum of $R\subset C^n$ with the image of each $B_i$ under an appropriate choice of rotation above, we obtain the desired manifold $R^*$. By construction, $R^*$ is isotopic to $R$ and, moreover, the orbit of $R^*$ under the above set of rotations gives a tiling of $\boxplus$.  A 2-dimensional tile $R^*$ created via cube swapping is shown in Figure~\ref{fig:2dschematic}. An example $R^*$ in dimension $n=3$ is shown in Figure~\ref{fig:3d-bw-reptile}, and the tiling of a cube by tiles isometric to $R^*$ is illustrated in Figure~\ref{fig:front-view-reptiles}. Finally, we show that this construction can be ``cubified", so that $R^*$ is a polycube tiling $\boxplus$, completing the proof of Theorem~\ref{tile_cube.thm}. Once this is established, Theorem~\ref{thm:main} follows from Lemma~\ref{mantra}.

\subsection{Taloned Patterns}\label{sec:talons}
We define the desired boundary pattern described above. A {\it $k$-claw} is a tree which consists of one central vertex $v$ and $k$ leaves, each connected to $v$ by a single edge. See Figure~\ref{fig:claw}.

\begin{figure}[htbp]
    \centering
    \includegraphics[width=2in]{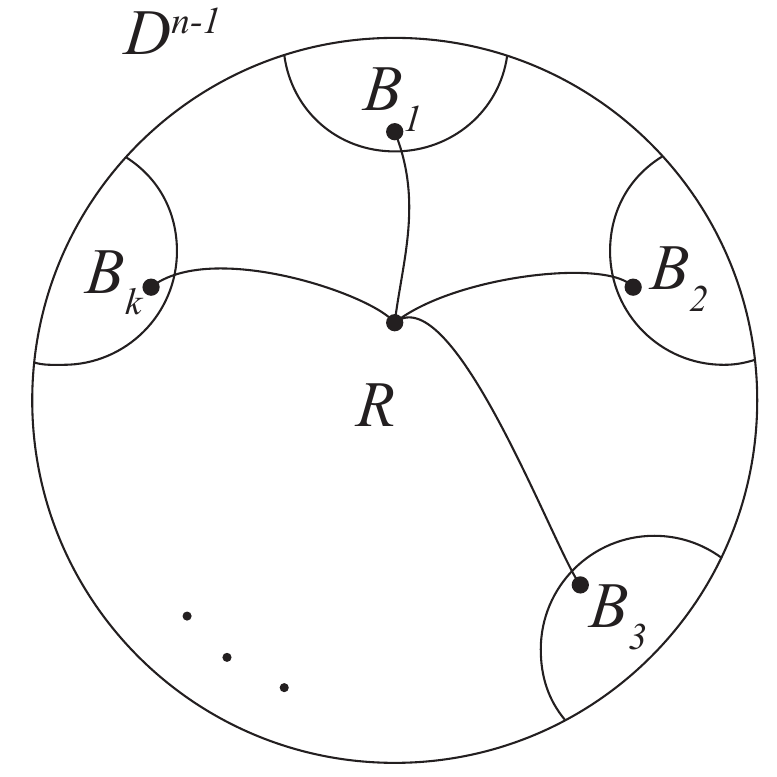}
    \caption{Taloned boundary pattern corresponding to a $k$-claw in $\partial C^n$. }
    \label{fig:claw}
\end{figure}

Our goal is to construct a boundary pattern on $C^n$ such that there exists an embedded disk $D^{n-1}\subset \partial C^n$ with the following properties:

\begin{itemize}
    \item $D^{n-1}\cap B_i$ is a single $(n-1)$-disk, for all $1\leq i \leq k$;
    \item $(D^{n-1}\cap B_i)\cap \partial D^{n-1}$ is an $(n-2)$-disk, for all $1\leq i \leq k$;
    \item $D^{n-1}\setminus (\cup_{i=1}^k B_i\cap D^{n-1})\subset R$.
    \item $B_i\cap B_j \cap D^{n-1}=\emptyset$ for $i\neq j$.
\end{itemize}

We regard the boundary pattern as the regular neighborhood of a $k$-claw, with the following decomposition: $R$ contains a neighborhood of the central vertex; and each $B_i$ containing a neighborhood of a leaf. See Figure~\ref{fig:claw}. We call this a {\it taloned pattern} of intersections.

We begin by proving Lemma~\ref{connected_bdry_v2.lem}, which ensures that, in the interior of $C^n$, the union of the boundaries of the pieces $\{R, B_1,\dots, B_n\}$ in the interior of our decomposition of $C^n$ can be assumed to be connected.

\begin{lemma}Let $R$ be a compact $n$-manifold with a single boundary component embedded in the $n$-cube $C^n$ such that $C^n=R \cup B_1\cup \dots \cup B_k$, where each $B_i$ is an $n$-ball, and such that the interiors of $R$ and the $B_i$ are pairwise disjoint.  Then after a homotopy of $C^n$ which restricts to isotopies on the interiors of $R$ and the $B_i$,  $W=\overline{\left(\partial  R\cup \partial B_1\cup \dots \cup \partial B_k\right)\setminus \partial C^n}$ is a connected $(n-1)$-complex.\label{connected_bdry_v2.lem} 
\end{lemma}

\begin{proof}

Let $\mathcal{B}=\{R, B_1, \dots , B_k\}$.  We partition $\mathcal{B}$ into {\it layers} $\mathcal{L}_i$ as follows (see Figure~\ref{fig:connecting_W}). Define the first layer as $\mathcal{L}_1=\{L\in \mathcal{B}~|~\partial L\cap \partial C^n \neq \emptyset\}$. We will use the notation $\partial \mathcal{L}_1:=\bigcup_{L\in \mathcal{L}_1} \partial L$.   Next choose a minimal collection of disjoint, embedded paths $\alpha_1,\dots \alpha_l$ on $\partial C^n$ such that \begin{itemize}\item 
$\left( \overline{\partial \mathcal{L}_1\setminus \partial C^n} \right)\cup \alpha_1\cup\dots \cup \alpha_l$ is connected,
\item the interior of $\alpha_i$ is contained in a single element $B(\alpha_i)$ of $\mathcal{B}$; and 
\item no $\alpha_i$ has both endpoints on the same connected component of  $\overline{\partial \mathcal{L}_1\setminus \partial C^n}$.

\end{itemize}  
\begin{figure}
    \centering
    \includegraphics[width=\textwidth]{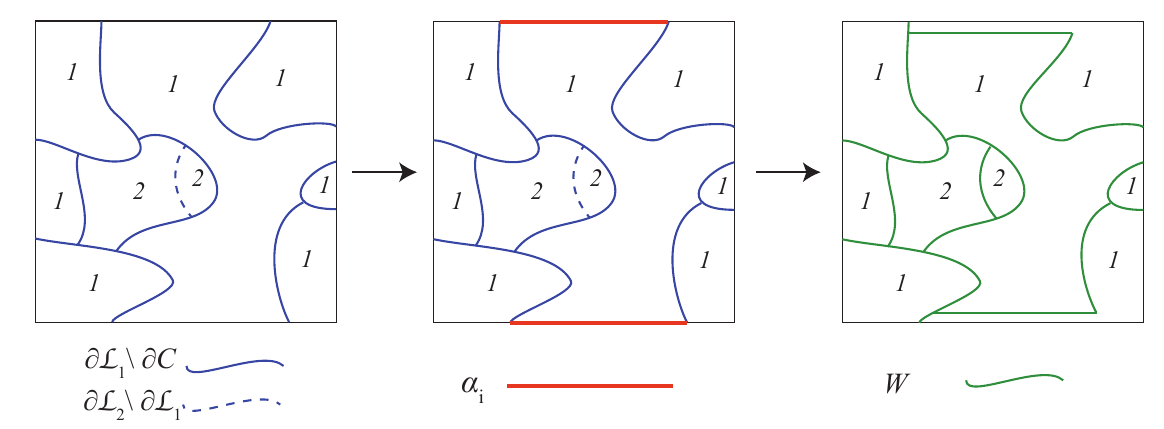}
    \caption{(Left) A partition of the decomposition $\{R,B_1,\dots,B_k\}$ of $C^n$ into layers, with the number in reach region indicating its level; (Middle) A choice of paths $\alpha_i$  on $\partial C^n$ such that after performing finger moves along the $\alpha_i$, $W=\overline{\left(\partial  R\cup \partial B_1\cup \dots \cup \partial B_k\right)\setminus \partial C^n}$ is connected (Right). } 
    \label{fig:connecting_W}
\end{figure}
\begin{figure}
\centering
\includegraphics[width=\textwidth]{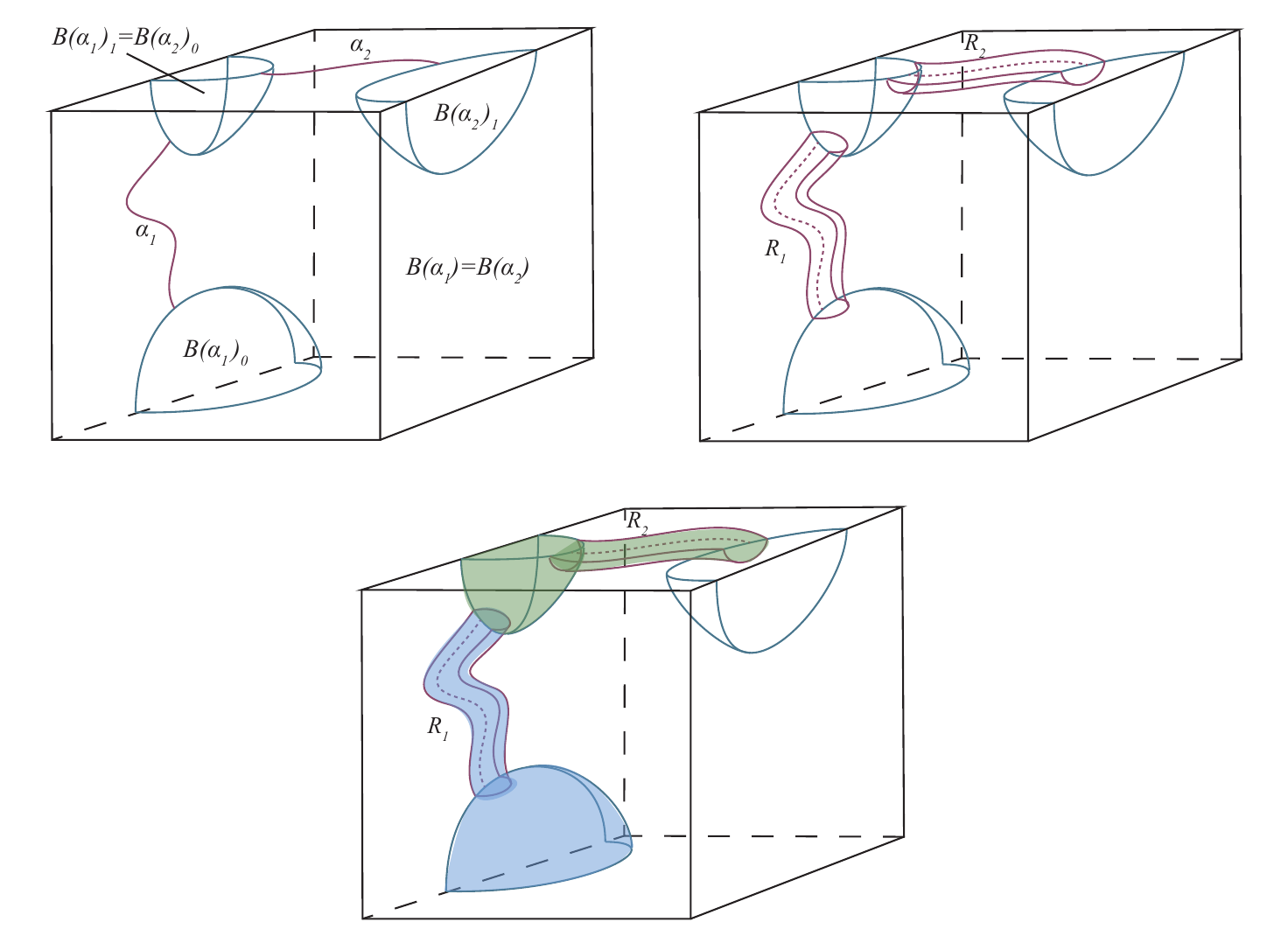}
\caption{Performing finger moves to ensure $W=\overline{\left(\partial  R \cup \partial B_1\cup \dots \cup \partial B_k\right)\setminus \partial C^n}$ is a connected $(n-1)$-complex. }
\label{fig:3d_finger_moves}
\end{figure}

Note that any given element of the decomposition $\mathcal{B}$ may contain the interior of more than one of the paths $\alpha_i$, i.e., it is possible to have $B(\alpha_i)=B(\alpha_j)$ for $i\neq j$. For each $A\in\mathcal{B}$, we let $P(A)$ denote the set of all $i$ such that $A=B(\alpha_i)$.

Since the $\alpha_i$ are disjoint, for each $1\leq i \leq l$, we can choose a disjoint regular neighborhood $R_i$ in $B(\alpha_i)$ of $\alpha_i$ such that $R_i$ intersects the boundary of exactly two other elements $B(\alpha_i)_0$ and $B(\alpha_i)_1$ of $\mathcal{B}$, one at each of the endpoints $\alpha_i(0)$ and $\alpha_i(1)$, respectively.  For each $A\in \mathcal{B}$, let $P_0(A)$ denote the set of all $i$ such that $A=B(\alpha_i)_0$. Next, modify the decomposition $\mathcal{B}$ of $C^n$ as follows (see Figure~\ref{fig:3d_finger_moves}).

\begin{itemize}
\item For each $A\in\mathcal{B}$, delete all the $R_i$ whose interiors intersect $A$, replacing each $A\in \mathcal{B}$ by $$A'=\overline{A\setminus \left(\bigcup_{i\in P(A)} R_i\right)}$$

\item Then, attach each $R_i$ to $B(\alpha_i)_0$, replacing each $A'$ (which may coincide with $A$, if $A$ did not intersect the interior of any $R_i$) by
$$A''=A'\cup \left(\bigcup_{i\in P_0(A)} R_i\right) $$
\end{itemize}

This process can be achieved by a homotopy of $C^n$ which restricts to isotopies on the interiors of the elements of $\mathcal{B}$. We imagine elements of $\mathcal{B}$ as growing fingers along the $\alpha_i$. From now on, we will simply call these {\it finger moves}  and will not describe them explicitly.

After performing finger moves on the elements of $\mathcal{L}_1$ along the $\alpha_i$, we can assume $\partial \mathcal{L}_1\setminus \partial C^n$ is connected.  Then inductively define  $\mathcal{L}_i=\{L\in \mathcal{B}\setminus \bigcup_{j=1}^{i-1}\mathcal{L}_{j}|L\cap \partial \mathcal{L}_{i-1}\neq\emptyset\}$, where $\partial \mathcal{L}_{i}$ is defined analogously to $\partial\mathcal{L}_{1}$.  
Since $\partial L$ is connected for each $L\in\mathcal{L}_2$ and meets $\partial \mathcal{L}_1\setminus \partial C^n$, we have that $\partial{\mathcal{L}}_2\cup (\partial{\mathcal{L}}_1\setminus \partial C^n)$ is connected. Continue inductively for each $3\leq i\leq m$, where $m$ is the number of layers.  By construction, $\partial L$ is connected  for each $L \in \mathcal{L}_i$ and intersects $\bigcup_{j=1}^{i-1} \partial\mathcal{L}_j\setminus \partial C^n$ non-trivially.  Therefore $\bigcup_{i=1}^m \partial\mathcal{L}_i\setminus \partial C
^n =\left(\partial  R\cup \partial B_1\cup \dots \cup \partial B_k\right)\setminus \partial C^n$ is connected, so 
$W=\overline{\left(\partial  R\cup \partial B_1\cup \dots \cup \partial B_k\right)\setminus \partial C^n}$
is connected as well.
\end{proof}

\begin{lemma}\label{lem:ClawsExist} Let $R$ be a compact $n$-manifold with connected boundary embedded in the $n$-cube $C^n$ such that $C^n=R \cup B_1\cup \dots \cup B_k$, with $k\leq n$, where each $B_i$ is a $n$-ball and such that the interiors of $R$ and the $B_i$ are pairwise disjoint.  After applying a self-homotopy of $C^n$ that restricts to an isotopy on the interior of each component in the above decomposition, we can find an $k$-claw embedded in $\partial C^n$ such that its regular neighborhood in $\partial C^n$ is a taloned pattern. 
    
\end{lemma}

\begin{proof}

By Lemma~\ref{connected_bdry_v2.lem}, we can assume $W=\overline{\left(\partial  R \cup \partial B_1\cup \dots \cup \partial B_k\right)\setminus \partial C^n}$ is connected, so we can perform a finger move on $R$ along a path in $W$ to ensure that $R$ meets $\partial C^n$. Since $R$ has a single boundary component and $(\partial R\cap \partial C^n) \subsetneq \partial C^n$, we can assume there exists a point $p$ on the interior of an $(n-1)$ face of $C^n$ that lies on $\partial R\cap\partial B_i$ for some $i$.  Relabeling the $B_i$ if necessary, we assume $i=1$.

Without loss of generality, assume $p$ lies on the face $F_1$ defined by $\{x_1=0\}\cap C^n$, and let $p=(0,p_2,\dots, p_n)$. After an isotopy of $C^n$, we can assume some $\epsilon$-ball $B_\epsilon(p)$ satisfies the following:

$$R\cap B_\epsilon(p)=\{(x_1,\dots, x_n)\in C^n\cap B_\epsilon(p)|x_2\geq p_2\}$$
$$B_1\cap B_\epsilon(p)=\{(x_1,\dots, x_n)\in C^n\cap B_\epsilon(p)|x_2\leq p_2\}.$$
We can further assume that $R\cap B_1\cap B_\epsilon(p)=W\cap B_\epsilon(p)$ and $R\cap B_1\cap B_\epsilon(p)=\{(x_1,\dots, x_n)\in C^n\cap B_\epsilon(p)|x_2=p_2\}$.

Choose distinct points $q_2,\dots,q_{k}$ on the $(n-2)$ disk $R\cap B_1\cap B_\epsilon(p)\cap \partial C^n$, as in Figure~\ref{fig:boundary_isotopy}.  We claim that one can choose disjoint paths $\delta_i\subset W$ from a point $r_i$ in $B_i\cap int(C^n)$ to the point $q_i$ for each $2\leq i \leq k$.

To produce the $\delta_i$, we again apply Lemma ~\ref{connected_bdry_v2.lem}.  In dimensions 4 and higher, we can achieve disjointness of the $\delta_i$ by a perturbation. In dimension 3, we perform an oriented resolution at each point of intersection of the $\delta_i$'s  which can not be removed by perturbation inside $W$. In dimension 2, there is only one such path, $\delta_2$, since $2\geq i\geq k=2.$

Once the paths are disjoint, we perform a finger move which pushes a neighborhood of $r_i$ in $B_i$ along $\delta_i$ to a neighborhood of $q_{i}$ in $B_\epsilon(p)$. As a result, the balls $B_i$ intersect $\partial C^n\cap B_\epsilon(p)$ in the boundary pattern shown in Figure~\ref{fig:boundary_isotopy} (middle).  We then choose a claw as shown in Figure~\ref{fig:boundary_isotopy} (bottom).  The  regular neighborhood of this claw in $\partial C^n$ is isotopic to a taloned pattern (Figure~\ref{fig:claw}), as desired. \end{proof}

\begin{figure}
    \centering
    \includegraphics[width=3in]{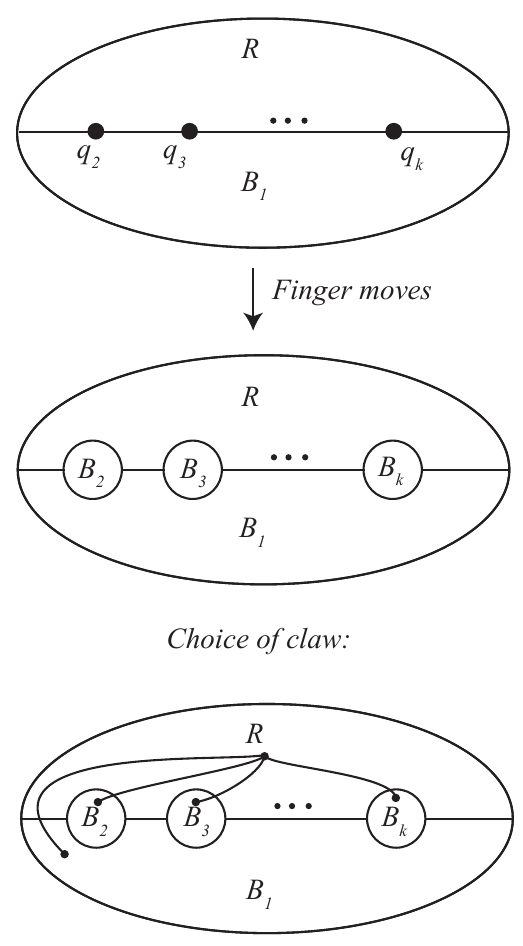}
    \caption{Three views of  $\partial C^n\cap B_\epsilon(p)$, showing the stages of obtaining the claw.  First, perform finger moves so that each ball $B_i$, with $i\geq 2$, meets $\partial C^n$ along the $(n-2)$-disk of intersection of $R$ and $B_1$ inside $B_\epsilon(p)$.  One can then choose a $k$-claw (bottom) which has a small regular neighborhood in $\partial C^n$ giving a taloned pattern of intersection.  }
    \label{fig:boundary_isotopy}
\end{figure}
\vspace{.3cm}

\subsection{Proof of main theorem.} \label{sec:rotation}

We begin by setting up the necessary notation. For each $i=1, \dots, n$, let $F_i$ be the $(n-1)$-dimensional face of the $n$-cube $C^n$ contained in the hyperplane $x_i=0$. For the moment, we will assume that $n$ is even. The case of $n$ odd requires an extra step, which we leave until the end of the proof. 

Let $r_i: \mathbb{R}^n \to \mathbb{R}^n$ be the rotation by $\frac{\pi}{2}$ about the $(n-2)$-plane $x_{2i-1}= x_{2i}=0$ that carries the $x_{2i-1}-$axis to the $x_{2i}-$axis. Note that each $r_i$ has order four and that these rotations commute, generating a group isomorphic to $(\mathbb{Z}_4)^{n/2}$. Given a vector $\mathbf{y} = (y_1, \dots, y_{n/2}) \in (\mathbb{Z}_4)^{n/2}$,  we define the rotation ${\bf r_y}$ as follows:
$${\bf r_y}=r_{n/2}^{y_{n/2}} \circ \dots \circ r_1^{y_1}.$$
We set $C_\mathbf{y}: =\mathbf{r_y}(C^n).$

We claim that the orbit of a unit sub-cube under this group action is the entire $n$-dimensional cube $\boxplus := [-1,1]^n$. In other words, $\boxplus$ is tiled by the $2^n$ distinct unit cubes $\{\mathbf{r_y}(C^n)|\mathbf{y} \in (\mathbb{Z}_4)^{n/2}\}.$  

To see this, first decompose $\boxplus$ into $2^n$ unit sub-cubes of the form $J_1 \times \dots \times J_n$, where each $J_i$ is either $[-1,0]$ or $[0,1]$. Fixing $k$, for each choice of $J_{2k-1}$ and $J_{2k}$ from the set $\{[-1,0], [0,1]\}$, the product $J_{2k-1}\times J_{2k}$ is a unit square in the $x_{2k-1}x_{2k}-$ plane, which we denote by $\mathbb{R}^2_k.$ Let $P_k:=C^n \cap \mathbb{R}^2_k$, i.e $P_k$ is the unit square in the first quadrant of $\mathbb{R}^2_k$.  Then $J_{2k-1}\times J_{2k}=r^{y_k}(P_k)$ for some $y_k\in\{0,1,2,3\}$. 
Hence, each of the $2^n$ unit cubes above can be expressed as
 $$J_1 \times \dots \times J_n = r_1 ^{y_1}(P_1) \times \dots \times r_{n/2} ^{y_{n/2}}(P_{n/2})= C_\mathbf{y}$$ for some $\mathbf{y} = (y_1, \dots, y_{n/2}) \in (\mathbb{Z}_4)^{n/2}$. Moreover, for each $J_1 \times \dots \times J_n$, the $\mathbf{y}$ such that $J_1 \times \dots \times J_n = C_\mathbf{y}$ is unique. To see this, note that each $J_1 \times \dots \times J_n$ has exactly one corner with all nonzero coordinates (and therefore with all coordinates $\pm 1$). On the other hand, the cube $C_\mathbf{y}$ also has exactly one corner $(c_1,\dots, c_n)$ with all $c_i=\pm 1$ (namely, the image of the point $(1, 1, 1, \dots, 1)\in C^n$), and its coordinates satisfy the formula $y_k=-(c_{2k}-1)-\frac{1}{2}(c_{2k-1}c_{2k}-1).$ In other words, the coordinates $(c_1,\dots c_n)$ uniquely determine each component $y_k,$
 and therefore $\mathbf{y}$ itself.

Observe that the cube $r_k(C^n)$ intersects $C^n$ along its face $F_{2k-1}$, and the cube $r^{-1}_{k}(C^n)$ intersects $C^n$ along its face $F_{2k}$. Thus, each rotation $r_k$ gives a pairing of the faces of $C^n$. We use this pairing to carry out a cube swap as previously described. This will allow us to build the rep-tile $R^*$. 

\subsection{Realizing the taloned pattern on $\partial C^n$} \label{sec:bdry pattern}

We will now describe a homotopy of $C^n$ which restricts to an isotopy on the interiors of $R$ and the balls $B_1,\dots, B_n$. Our goal is to use Lemma~\ref{lem:ClawsExist} to position $R$ and $B_1,\dots, B_n$ so that their intersections with the boundary of $C^n$ satisfy: 

 \begin{enumerate} 
 \item \label{bdry_pattern_prop1} For each $1\leq i \leq n$, the only ball meeting the face $F_i$ is $B_i$ (and thus $F_i\setminus (B_i\cap F_i) \subset R$),  
 \item \label{bdry_pattern_prop2} $r_k(B_{2k}\cap F_{2k})\subset F_{2k-1}$ is disjoint from $B_{2k-1}$, and 
 \item \label{bdry_pattern_prop3} $r_k^{-1}(B_{2k-1}\cap F_{2k-1})\subset F_{2k}$ is disjoint from $B_{2k}$.
 \end{enumerate}

In what follows, we refer the reader to a schematic in Figure~\ref{fig:original_cube}.  Figure~\ref{fig:original_cube_4D} illustrates this configuration in dimension 4. 

For each $k = 1, \dots, \frac{n}{2}$, let $\varphi_{2k-1}$ be the $(n-2)$-facet in $C$ equal to the intersection of $C$ with the $(n-2)$-plane given by setting $x_{2k-1}=0$ and $x_{2k}=1$. Likewise, let $\varphi_{2k}$ be the $(n-2)$-facet in $C$ equal to the intersection of $C$ with the $(n-2)$-plane given by setting $x_{2k-1}=1$ and $x_{2k}=0$. Note that this pair of facets are exactly those that are simultaneously parallel to the intersection $F_{2k-1} \cap F_{2k}$ and contained in $F_{2k-1} \cup F_{2k}$. 

\begin{figure}[htbp]
    \centering
    \includegraphics[width=\textwidth]{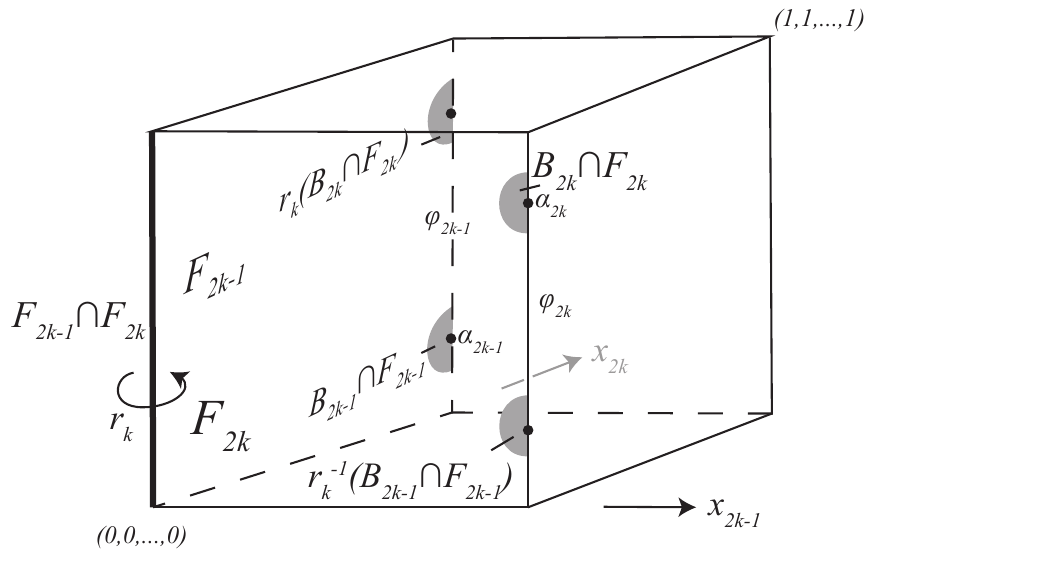}
    \caption{Intersections of $B_{2k-1}$ and $B_{2k}$ with faces $F_{2k-1}$ and $F_{2k}$ of $\partial C^n$, and their images under the rotations $r_k^{-1}$ and $r_k$,  respectively. }
    \label{fig:original_cube}
\end{figure}

\begin{figure}[htbp]
    \centering
    \includegraphics[width=\textwidth]{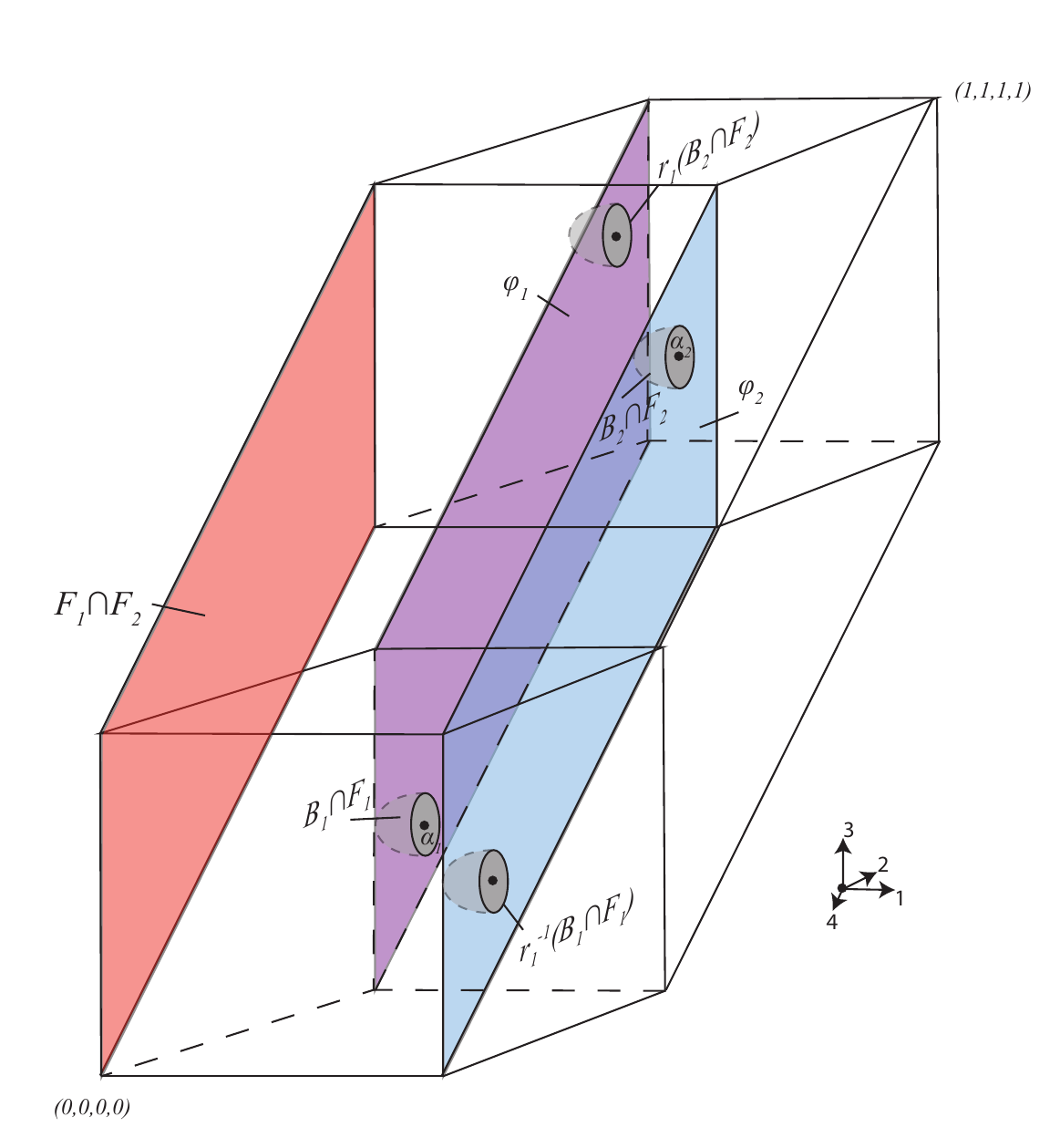}
    \caption{Intersections of $B_1$ and $B_2$ with faces $F_1$ and $F_2$ of $\partial C^4$, and their images under the rotations $r_1^{-1}$ and $r_1$,  respectively.}
    \label{fig:original_cube_4D}
\end{figure}

Now, we fix points $\alpha_{2k-1} \in \varphi_{2k-1}$ and $\alpha_{2k} \in \varphi_{2k}$ by setting $$\alpha_{2k-1} = \left(\frac{1}{4}, \dots, \frac{1}{4}, 0 ,1, \frac{1}{4}, \dots, \frac{1}{4}\right)$$ and $$\alpha_{2k} = \left(\frac{3}{4}, \dots, \frac{3}{4}, 1, 0, \frac{3}{4}, \dots, \frac{3}{4} \right),$$ where the $0$ and $1$ entries are taken to be in the $(2k-1)^{st}$ and $2k^{th}$ coordinates. 

 Let $B_1, B_2, \dots, B_n$ be the $n$-balls whose existence is guaranteed by Theorem \ref{theorem:ballnumberboundary}. By Lemma \ref{lem:ClawsExist}, after an isotopy of $R$ and the $B_i$, there is an $n$-claw embedded in $\partial C^n$ such that its regular neighborhood in $\partial C^n$ is a taloned pattern as shown in Figure~\ref{fig:claw}. Moreover, after an isotopy of $C^n$ supported near its boundary, we can assume that the taloned pattern is mapped homeomorphically to $\bigcup_{i=1}^{n} F_i$ such that the intersection $F_i \cap B_i:=N_i$ is a closed regular neighborhood of radius $1/8$ of the point $\alpha_i$ in $F_i$, and also that if $i \not= j$, then $F_i \cap B_j = \emptyset$. We do not assume any restrictions on the intersections of $R$
 and the $B_i$ with the remaining faces $x_i=1$ of~$C^n$.

Note that this set-up has several convenient consequences. First, the union $\bigcup_{i=1}^{n} F_i$ intersects $\partial R$ in a single $(n-1)$-ball, since $R$ meets the taloned pattern in a single $(n-1)$-ball.  
Furthermore, the center and radius of $N_{2k}$ were chosen to guarantee that the ball $r_k(N_{2k})\subset F_{2k-1}$ is disjoint from the neighborhood $N_{2k-1}$, and therefore contained in $F_{2k-1}\backslash N_{2k-1} =R \cap F_{2k-1}$. Similarly, the ball $r_k^{-1}(N_{2k-1})$ is contained in $F_{2k}\backslash
N_{2k}= R \cap F_{2k}$.

\subsection{Cubification of the decomposition}  \label{sec:cubification}
Recall that for any positive integer $m$, by $\mathcal{C}(\mathcal{Z}_{\frac{1}{m}}^n)$ we denote the lattice in $\mathbb{R}^n$ whose unit cubes have side length $\frac{1}{m}$.

Let $W=\overline{(\partial R \cup \partial B_1 \cup \dots \cup \partial B_n)\setminus \partial C^n}$. Since $R$ and each $B_i$ can be assumed piecewise-smooth, $W$ has a closed regular neighborhood $N(W)$. Being a codimension-0 compact submanifold of $\mathbb{R}^n,$ it is isotopic to a polycube, also denoted $N(W)$, in a sufficiently fine lattice $\mathcal{C}(\mathcal{Z}_{\frac{1}{m}}^n)$, by Proposition \ref{Prop:Smooth_to_poly}. (In the course of cubification, we shall increase $m$ as needed without further comment.) We also assume that all cubes in $N(W)$ which intersect $R$ form a regular neighborhood of $\partial R$. Similarly for each $B_i$;  and for each double intersection, $\partial B_i\cap \partial B_j$ or $\partial R\cap \partial B_i$; and each triple intersection, etc.

The closure $\overline{R\backslash N(W)}$ is then also a polycube; similarly for each $\overline{B_i\backslash N(W)}$. To complete the cubification of the ensemble $\{R, B_1, \dots, B_n\},$ we assign cubes in $N(W)$ back to the constituent pieces in an iterative fashion. Specifically, all cubes in $N(W)$ which intersect $R$ are assigned to $R$, and their union is denoted $R^{cu}$; of the remaining cubes, all that intersect $B_1$ are assigned to $B_1$, and the resulting polycube is denoted $B_1^{cu}$; and so on. By the above assumptions, each of the pieces $\{R, B_1, \dots, B_n\}$ is isotopic to the corresponding polycube since we are only adding or removing small cubes intersecting the boundary. In addition, the union of the interiors of $\{R, B_1, \dots, B_n\}$ is isotopic to the union of the interiors of $\{R^{cu}, B_1^{cu}, \dots, B_n^{cu}\}$.

Furthermore, by selecting a sufficiently fine lattice, we can ensure that the isotopies performed, taking each of $\{R, B_1, \dots, B_n\},$ to a polycube, are arbitrarily small. Thus, they preserve properties (1), (2) and (3) from Section~\ref{sec:bdry pattern}.

Recycling notation, we will from now on refer to $R^{cu}$, $B^{cu}_1$, $\dots$, $B^{cu}_n$ as $R$, $B_1$, $\dots$, $B_n$ respectively.

\vspace{-.1cm}

\subsection{Construction of the rep-tile.} Finally we construct our rep-tile $R^* \subset \boxplus$:

$$R^*=R \cup \left(\bigcup_{k=1}^{n/2} r_k^{-1}(B_{2k-1})\cup r_k(B_{2k})\right)$$

Recall that, at this stage of the construction, $R$ and all the $B_j$ are polycubes, and therefore so is $R^\ast$. We claim that {\bf(1)} $R^*$ is isotopic to $R$, and {\bf(2)} $2^n$ isometric copies of $R^*$ tile the cube $\boxplus$. A schematic of $R^*$ in dimension $n=3$ is shown in Figure ~\ref{fig:3dcubesschematic_adjacentonly} (for intuition in the case of $n$ even, simply ignore $B_3$ and its rotated copy in the figure).

%This figure is from a 3d model available at [cite the github], where the reader can examine the rep-tile from all sides

\vspace{3mm}
\begin{figure}[htbp]
    \centering
    \includegraphics[width=\linewidth]{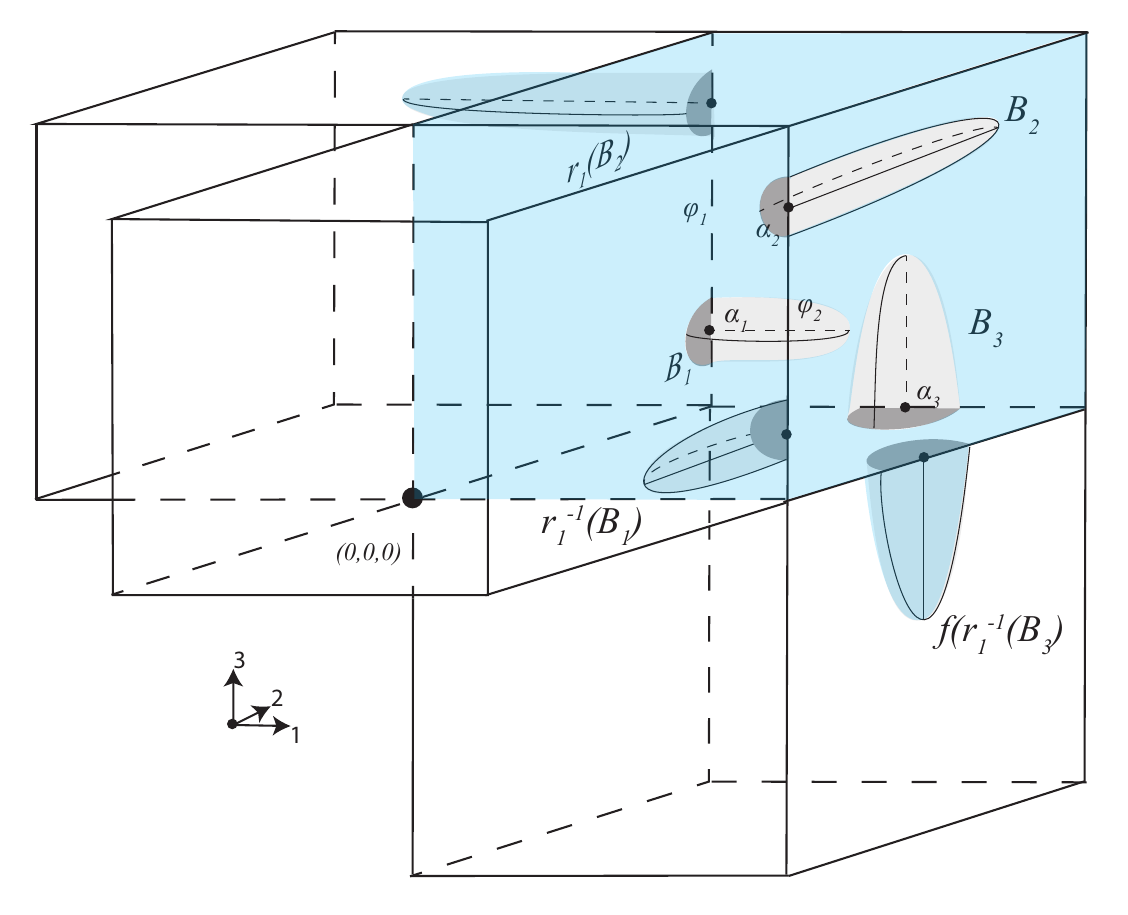}
    \caption{Schematic of the construction of $R^*$, shown in blue. In this picture, the unions of cubes which undergo cube swaps are drawn as balls.}
    \label{fig:3dcubesschematic_adjacentonly}
\end{figure}

\begin{figure}[htbp]
    \centering
    \includegraphics[width=2in]{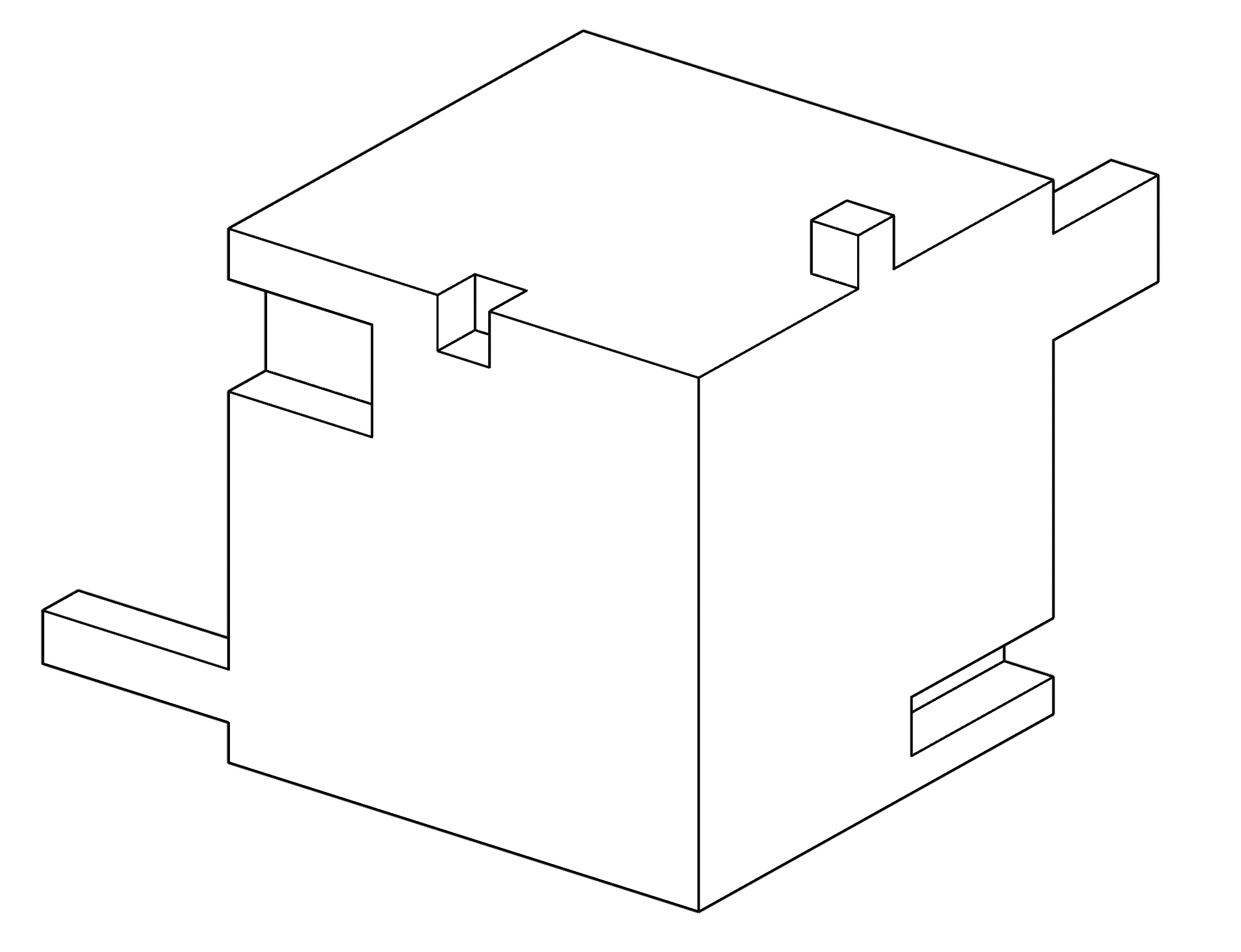}
    \caption{Example of a rep-tile $R^\ast$ obtained by cube swapping.} %The rep-tile has been rotated for better visibility of $B_3$ and $f(r_1^{-1}(B_3))$.}
    \label{fig:3d-bw-reptile}
\end{figure}

\begin{figure}[h]
    \centering
    \includegraphics[width=2.25in]{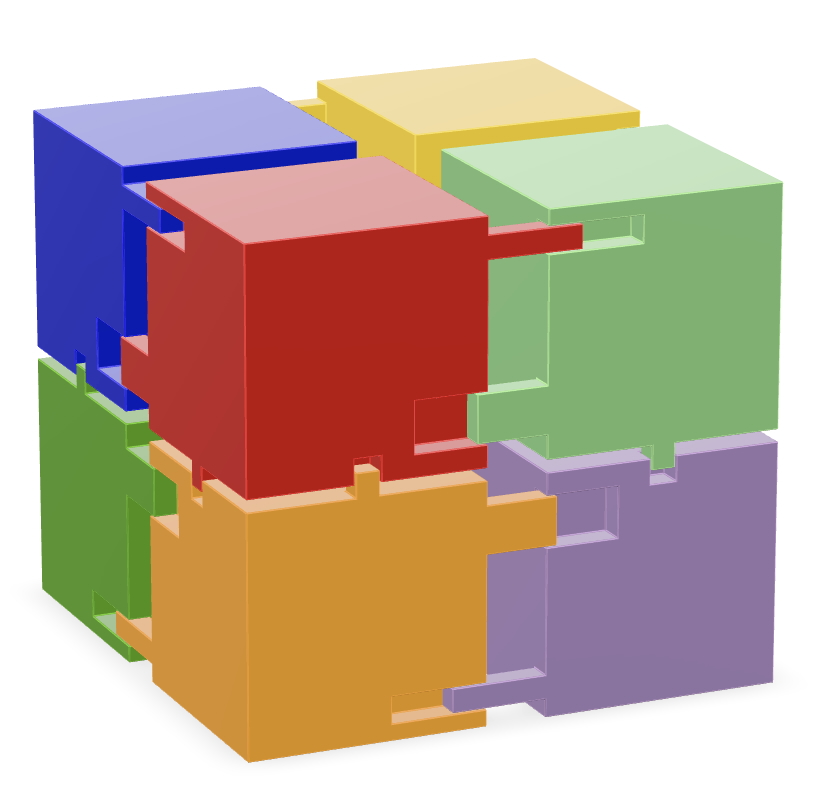}
    \caption{Eight copies of $R^\ast$ tiling the cube. This figure is from a 3D model available at \cite{Geissler3D}, where the reader can rotate the tiled cube and examine the rep-tile from all sides.} 
    \label{fig:front-view-reptiles}
\end{figure}

%\begin{figure}[h]
 %   \centering
  %  \includegraphics[width=2.25in]{figures/rep-tiles-front.jpg}
   % \caption{Three copies of $R^\ast$ on their way to tiling the cube.} 
   % \label{fig:front-view-reptiles}
%\end{figure}

%\begin{figure}[h]
 %   \centering
  %  \includegraphics[width=2.25in]{figures/rep-tiles-back.jpg}
   % \caption{Same three copies of $R^\ast$ viewed from the opposite side, where a fourth copy of $R^\ast$ (not pictured) fits in. } 
   % \label{fig:back-view-reptiles}
%\end{figure}

\emph{Proof of (1).} 
The images of the cube $C^n$ under the rotations $r_1, r_1 ^{-1}, \dots, r_{n/2}, r_{n/2}^{-1}$ give a family of $n$ distinct unit cubes in $\boxplus$, each of which shares a unique face with $C^n$. More specifically, the cube $r_k(C^n)$ intersects $C^n$ along its face $F_{2k-1}$, and the cube $r_k^{-1}(C^n)$ intersects $C$ along its face $F_{2k}.$  Refer to Figure~\ref{fig:original_cube}.

It follows that the intersections of each ball $r_k(B_{2k})$ and $r_k^{-1}(B_{2k-1})$ with the cube $C^n$ are disjoint $(n-1)$-balls contained in $\partial R \cap C^n$. (Recall that the center and radius of the $N_i$ were chosen carefully so that this is the case.) Therefore, $R^*$ is a boundary connected sum of $R \subset C^n$  with a collection of $n$-balls, one in each neighboring cube. An isotopy therefore brings $R^*$ to the initial embedding of $R$, as desired. This concludes the proof of (1).

\vspace{3mm}

\emph{Proof of (2).} Let $R_{\bf y}: = {\bf r_y}(R)$, $B_{i\,{\bf y}}:= {\bf r_y}(B_i)$, and $R^*_{\bf y}: = {\bf r_y}(R^*)$. Note that, since $C^n=R \cup(\cup_i B_i)$, we have that  $R_{\bf y}\subseteq C_{\bf y}$ and $B_{i\,{\bf y}}\subseteq C_{\bf y}$. In addition, the first equality on the next line clearly implies the second:
\[
\boxplus=\bigcup_{\mathbf{y}\in (\mathbb{Z}_4)^{n/2}} C_{\bf y}= \left(\bigcup_{\mathbf{y}\in (\mathbb{Z}_4)^{n/2}} R_{\bf y}\right)\cup \left(\bigcup_{\mathbf{y}\in (\mathbb{Z}_4)^{n/2} }(\cup_i B_{i\,{\bf y}}) \right).
\]
We now show that $$\boxplus=\bigcup_{\mathbf{y}\in (\mathbb{Z}_4)^{n/2}} R^*_{\mathbf{y}}.$$ Since $\boxplus$ decomposes into the cubes $C_{\mathbf{y}}$, it is sufficient to show that every point $p \in C_{\mathbf{y}}$ is contained in $R^*_{\bf v}$ for some ${\mathbf{v}\in (\mathbb{Z}_4)^{n/2}}$. This is a consequence of the fact that $R^*$ is the union of $R$ and one ball from the orbit of $B_i$ for each $i$. However, this fact may not be self-evident, so we provide an explicit proof.

Consider a point $p \in C_{\mathbf{y}}$. If $p$ is in the orbit of $R$, then $p \in R_{\mathbf{y}} \subset C_{\mathbf{y}}$, so $p \in R_{\bf y} \subseteq R^*_{\mathbf{y}}$. Now, suppose $p\in B_{i\,{\bf y}}$ for some $i=1, \dots n$. To find which rotation of $R^*$ contains $p$, consider the isometric ball $B_i \subset C^n$. There are two cases: if $i=2k-1$, then $B_i \subset r_k(R^*)$, and if $i=2k$, $B_i \subset r_k^{-1}(R^*)$.  

Let $\mathbf{v} \in (\mathbb{Z}_4)^{n/2}$ be the vector with $r_{\mathbf{v}}$ equal to $r_{\mathbf{y}}\circ r_k$ if $i=2k-1$ and $r_{\mathbf{y}}\circ r_k^{-1}$ if $i=2k$.  In other words, the vector $\mathbf{v}$ is equal to the vector $\mathbf{y}$ modified only by shifting its $k^{th}$ coordinate by $\pm 1$. Observe that $(B_i)_\mathbf{y} \subset (R^*)_\mathbf{v}$. This shows that $\boxplus$ indeed is equal to the union of the $R^*_\mathbf{y}$. 

To show that $\boxplus$ is \emph{tiled} by isometric copies of $R^*$, we need to check that the $R^*_\mathbf{y}$ have non-overlapping interiors. First observe that $R^*$ has $n$-volume 1, and that $\boxplus$ has $n$-volume $2^n$. Since exactly $2^n$ isometric copies of $R^*$ make up $\boxplus$, they must have disjoint interiors.  This concludes the proof of (2).

\subsection{Constructing the rep-tile in odd dimensions}
We have yet to handle the case where $n$ is odd, i.e. $n = 2m+1$ for some integer $m>0$. As before, let $F_i$ denote the face of $C^n$ intersecting the $(n-1)$-plane where $x_i = 0$. In this case, in addition to the rotations $r_1, \dots, r_m$ defined above, we require an additional rotation $f: \mathbb{R}^n \to \mathbb{R}^n$ by an angle of $\pi$ about the $(n-2)$-plane where $x_{n-1} = 0 = x_ n$. Note that by definition, $F_n = f \circ r_{(n-1)/2}^{-1}(F_n)$, and so $f \circ r_{(n-1)/2}^{-1}$ carries $C^n$ to its $n^{th}$ neighboring cube in $\boxplus$. 

For $i=1, \dots, n-1$, choose points $\alpha_i \in F_i$ as before. Choose the point $\alpha_n$ on the $(n-2)$-facet of $F_n$ where $F_n$ intersects the $(n-1)$-plane $x_{n-1}=1$.  More specifically, we let 
$$\alpha_{n} = \left(\frac{1}{2}, \dots, \frac{1}{2}, 1 ,0 \right)$$ 
and $N_n$ be a neighborhood of $\alpha_n$ in the face $F_n$ with radius $1/8$. This guarantees that $f\circ r_{(n-1)/2}^{-1} (N_n)$ is disjoint from $N_n$. Therefore, we can again define the boundary sum:

$$R^*=R \cup \left( \bigcup_{k=1}^{m} r_k^{-1}(B_{2k-1})\cup r_k(B_{2k})\right) \cup \left(f\circ r_{(n-1)/2}^{-1}(B_n)\right),$$

\noindent which is isotopic to $R$ and tiles $\boxplus=[-1,1]^n$ as before. To complete our proof that $R^*$ is a rep-tile for any $n$, we appeal to Lemma \ref{mantra}. 
\qed 

\subsection{All is non-rep-tile}

In this section we show that every smooth compact $n$-manifold in $\mathbb{R}^n$ is isotopic to a submanifold of $\mathbb{R}^n$ that is not a rep-tile. This shows that Theorem \ref{thm:main} is best possible in the sense that manifolds may satisfy the hypotheses of Theorem~\ref{thm:main} yet fail to be be rep-tiles, unless an isotopy is applied. In fact, a refinement of the below result would show that ``most'' manifolds in the isotopy class of a rep-tile are not themselves rep-tiles.

\begin{proposition}\label{Prop:best}
    Every smooth $n$-dimensional submanifold of $\mathbb{R}^n$ is topologically isotopic to a submanifold that does not tile $\mathbb{R}^n$.
\end{proposition}

\begin{proof}
Given a connected $n$-dimensional polycube $X$, we will refer to the constituent $n$-cubes of $X$ as {\it unit} cubes and we will say that a unit $n$-cube $C$ in $X$ is a \emph{peninsula} if it intersects the other {\it unit} cubes of $X$ along exactly one face.

Suppose $M$ is a smooth $n$-dimensional submanifold of $\mathbb{R}^n$. By Proposition \ref{Prop:Smooth_to_poly}, we can isotope $M$ to be an $n$-dimensional polycube $X$ made of (sufficiently small) unit $n$-cubes and assume that $X$ contains $k$ such cubes. Subdivide every unit $n$-cube in $X$ into $3^n$ subcubes creating a $n$-dimensional polycube $X’$ that is the union of $(3^n)k$ cubes, but is equal to $X$ as a set. Note that $X'$ cannot contain any peninsulas. Let $C$ be a unit $n$-cube of $X$ that meets $\partial X$ in a face $F$. Let $C'$ be the $n$-cube of side length $
\frac{1}{3}$ in $X’$ that meets the center of $F$. Then $X''=\overline{X'\setminus C'}$ is an $n$-dimensional polycube containing $(3^n)k-1$ $n$-cubes, and $X''$ is isotopic to $M$. 

Since $X’$ has no peninsulas, the only peninsulas for $X''$ must be contained in the $n$-cube $C$. However, if $n\geq 3$, each of the $n$-cubes of $X''$ in $C$ meet at least 2 other cubes along faces.  Hence, if $n\geq 3$, then $X''$ has no peninsulas. In the case when $n=2$, to ensure $X''$ has no peninsulas we must choose $C$ so that $C$ meets the boundary of $X$ in exactly one face $F$. We can ensure such a $C$ exists by first subdividing the initial $X$. In each case, $X''$ has no peninsulas.

Suppose that $X''$ tiles $\mathbb{R}^n$. Then there is an isometric copy of $X''$, denoted $X_{1}''$, that contains a cube $C_1$ that fills the hole created by the removal of $C'$ from $X'$. Thus, $C_1$ is a peninsula for $X_{1}''$, which is impossible. 
\end{proof}

Since every $n$-dimensional rep-tile tiles $\mathbb{R}^n$, an immediate consequence of the above proposition is that every smooth $n$-dimensional submanifold of $\mathbb{R}^n$ is topologically isotopic to a submanifold that is not a rep-tile.

\vspace{.5cm}

\subsection*{Acknowledgments: }This paper is the product of a SQuaRE. We are indebted to AIM, whose generous support and hospitality made this work possible. AK is partially supported by NSF grant DMS-2204349, PC by NSF grant DMS-2145384, RB by NSF grant DMS-2424734, and HS by NSF grant DMS-1502525. We thank Kent Orr for many helpful discussions. We are grateful to Richard Schwartz for his feedback on the first version of this paper; and for numerous valuable suggestions, notably a simplification of our original construction of spherical rep-tiles.

\section*{Ball Number}

\noindent{\it Let $R$ be a frog with a cube for a bride}

\noindent{\it Place $R$ in a box with some balls beside}

\noindent{\it Set free, the balls}

\noindent{\it Dance through walls}

\noindent{\it Out plops a Rep-tile with frogs inside}

\bibliographystyle{alpha}

\end{document}